\documentclass[reqno]{amsart}
\usepackage{amssymb,eucal,latexsym, mathrsfs,xy,enumerate,graphicx,url}

\xyoption{all}

{\end{enumerate}}
\newenvironment{enumeratea}{\begin{enumerate}[\upshape (a)]}%
{\end{enumerate}}
\newenvironment{enumerater}{\begin{enumerate}[\upshape (1)]}%
{\end{enumerate}}

\usepackage[usenames]{color}

\newcommand{\pup}[1]{\textup{(}{#1}\textup{)}}

\newcommand{\idproj}{i\-de\-al-pro\-jec\-tive}
\newcommand{\idprojy}{i\-de\-al-pro\-jec\-tiv\-ity}
\newcommand{\jirr}{join-ir\-re\-duc\-i\-ble}

\newcommand{\jh}{join-ho\-mo\-mor\-phism}
\newcommand{\mh}{meet-ho\-mo\-mor\-phism}

\newcommand{\xF}{\mathbf{F}}
\newcommand{\WD}{\mathbf{wD}}
\newcommand{\WDj}{\mathbf{wD}_{\vee}}
\newcommand{\Mnoeth}{\cM_{\textup{n{\oe}th}}}
\newcommand{\DX}{\mathsf{D}_{4}}
\newcommand{\hDX}{\widehat{\mathsf{D}}_{4}}

\DeclareMathOperator{\Sub}{Sub}
\DeclareMathOperator{\Ji}{Ji}
\DeclareMathOperator{\Id}{Id}
\DeclareMathOperator{\Pow}{Pow}

\newcommand{\gf}{\varphi}

\newcommand{\fine}[1]{[#1]_{\neq\es}^{<\omega}}

\newcommand{\go}{\omega}
\newcommand{\gl}{\lambda}
\newcommand{\gL}{\Lambda}
\newcommand{\gO}{\Omega}

\newcommand{\jz}{$(\vee,0)$}
\newcommand{\js}{join-semi\-lat\-tice}

\newcommand{\jzh}{\jz-ho\-mo\-mor\-phism}

\newcommand{\ol}[1]{\overline{#1}}

\newcommand{\pI}[1]{\bigl({#1}\bigr)}

\newcommand{\set}[1]{\left\{#1\right\}}
\newcommand{\setm}[2]{\set{{#1}\mid{#2}}}
\newcommand{\vecm}[2]{\left({#1}\mid{#2}\right)}
\newcommand{\seq}[1]{\langle{#1}\rangle}
\newcommand{\seqm}[2]{\langle{#1}\mid{#2}\rangle}

\newcommand{\sB}{{\mathsf{B}}}
\newcommand{\sP}{{\mathsf{P}}}
\newcommand{\sM}{{\mathsf{M}}}
\newcommand{\sN}{{\mathsf{N}}}
\newcommand{\id}{\mathrm{id}}

\newcommand{\dnw}{\mathbin{\downarrow}}

\newcommand{\upw}{\mathbin{\uparrow}}
\newcommand{\es}{\varnothing}
\newcommand{\res}{\mathbin{\restriction}}

\newcommand{\ZZ}{\mathbb{Z}}

\newcommand{\cC}{\mathcal{C}}
\newcommand{\cD}{\mathcal{D}}

\newcommand{\cL}{\mathcal{L}}
\newcommand{\cM}{\mathcal{M}}
\newcommand{\cN}{\mathcal{N}}
\newcommand{\cR}{\mathcal{R}}
\newcommand{\cRm}{\mathcal{R}_{\mathrm{mod}}}
\newcommand{\cU}{\mathcal{U}}
\newcommand{\cV}{\mathcal{V}}

\numberwithin{equation}{section}

\theoremstyle{plain}

\newtheorem{theorem}{Theorem}[section]
\newtheorem{proposition}[theorem]{Proposition}
\newtheorem{corollary}[theorem]{Corollary}
\newtheorem{lemma}[theorem]{Lemma}

\newtheorem{claim}{Claim}
\newtheorem*{sclaim}{Claim}

\theoremstyle{definition}

\newtheorem{definition}[theorem]{Definition}

\newtheorem{example}[theorem]{Example}
\newtheorem{problem}{Problem}

\theoremstyle{remark}
\newtheorem{remark}[theorem]{Remark}
\newtheorem*{note}{Note}

\newcommand{\qedc}{{\qed}~{\rm Claim~{\theclaim}.}}
\newcommand{\qedsc}{{\qed}~{\rm Claim.}}

\newenvironment{cproof}
{\begin{proof}[Proof of Claim.]}
{\qedc\renewcommand{\qed}{}\end{proof}}

\newenvironment{scproof}
{\begin{proof}[Proof of Claim.]}
{\qedsc\renewcommand{\qed}{}\end{proof}}

\numberwithin{figure}{section}
\numberwithin{table}{section}

\newcommand{\ba}{\boldsymbol{a}}
\newcommand{\bb}{\boldsymbol{b}}

\newcommand{\bx}{\boldsymbol{x}}

\newcommand{\va}{\mathsf{a}}
\newcommand{\vb}{\mathsf{b}}

\newcommand{\vx}{\mathsf{x}}
\newcommand{\vy}{\mathsf{y}}
\newcommand{\vu}{\mathsf{u}}
\newcommand{\vv}{\mathsf{v}}

\title[Projectivity and transferability]{Relative projectivity and transferability\\
for partial lattices}

\author[F. Wehrung]{Friedrich Wehrung}
\address{LMNO, CNRS UMR 6139\\
D\'epartement de Math\'ematiques\\
Universit\'e de Caen Normandie\\
14032 Caen cedex\\
France}
\email{friedrich.wehrung01@unicaen.fr}
\urladdr{http://www.math.unicaen.fr/\~{}wehrung}

\date{\today}

\begin{document}

\begin{abstract}
A partial lattice~$P$ is \emph{\idproj}, with respect to a class~$\cC$ of lattices, if for every~$K\in\cC$ and every homomorphism~$\gf$ of partial lattices from~$P$ to the ideal lattice of~$K$, there are arbitrarily large choice functions $f\colon P\to\nobreak K$ for~$\gf$ that are also homomorphisms of partial lattices.
This extends the traditional concept of (sharp)  \emph{transferability} of a lattice with respect to~$\cC$.
We prove the following:
\begin{enumerater}
\item A finite lattice~$P$, belonging to a variety~$\cV$, is sharply transferable with respect to~$\cV$ if{f} it is projective with respect to~$\cV$ and \emph{weakly distributive} lattice homomorphisms, if{f} it is \idproj\ with respect to~$\cV$.

\item Every finite distributive lattice is sharply transferable with respect to the class~$\cRm$ of all relatively complemented modular lattices.

\item The gluing~$\DX$ of two squares, the top of one being identified with the bottom of the other one, is sharply transferable with respect to a variety~$\cV$ if{f}~$\cV$ is contained in the variety~$\cM_{\go}$ generated by all lattices of length~$2$.

\item $\DX$ is projective, but not \idproj, with respect to~$\cRm$\,.

\item $\DX$ is transferable, but not sharply transferable, with respect to the variety~$\cM$ of all modular lattices.
This solves a 1978 problem of G. Gr\"atzer.

\item We construct a modular lattice whose canonical embedding into its ideal lattice is not pure.
This solves a 1974 problem of E. Nelson.
\end{enumerater}
\end{abstract}

\maketitle

\section{Introduction}\label{S:Intro}

\subsection{Background}\label{Su:Backgrnd}
A lattice~$P$ is \emph{transferable} (resp., \emph{sharply transferable}) if for every lattice~$K$ and every lattice embedding~$\gf$, from~$P$ to the lattice~$\Id K$ of all ideals of~$K$, there exists a lattice embedding $f\colon P\hookrightarrow K$ (resp., a lattice embedding $f\colon P\hookrightarrow K$ such that $f(x)\in\gf(y)$ if{f} $x\leq y$, whenever $x,y\in P$).
This concept, originating in Gaskill~\cite{Gask72}, started a considerable amount of work involving many authors, which led to the following characterization of (sharp) transferability for finite lattices (cf. pages~502 and~503 in Gr\"atzer~\cite{LTF} for details).

\begin{theorem}\label{T:CharTransf}
The following statements are equivalent, for any finite lattice~$P$:
\begin{enumerater}
\item\label{Ltransf}
$P$ is transferable;

\item\label{Lshtransf}
$P$ is sharply transferable;

\item\label{LbdedW}
$P$ is a bounded homomorphic image of a free lattice, and it satisfies Whitman's Condition;

\item\label{LSDW}
$P$ is semidistributive, and it satisfies Whitman's Condition;

\item\label{Lsubfree}
$P$ can be embedded into a free lattice;

\item\label{Lproj}
$P$ is projective.

\end{enumerater}
\end{theorem}

Theorem~\ref{T:CharTransf} does not extend to infinite lattices.
For example, the chain $\set{0,1}\times\go$ is transferable, but not sharply transferable (cf. Tan~\cite{Tan75}).
However, the implication \eqref{Lshtransf}$\Rightarrow$\eqref{Lproj} (i.e., sharp transferability implies projectivity) holds in general, see Nation~\cite{NationBFL} together with R. Freese's comments on page~588 in Gr\"atzer \cite[Appendix~G]{GLT2}.

Requiring the condition $K\in\cC$ in the definition above, for a class~$\cC$ of lattices, relativizes to~$\cC$ the concepts of transferability and sharp transferability.

For the variety~$\cD$ of all distributive lattices, more can be said.
The following result is implicit in Gaskill~\cite{Gask72a}, with a slightly incorrect proof in the first version of that paper.
It is also stated explicitly in Nelson~\cite{Nels74}.

\begin{theorem}[Gaskill, 1972; Nelson, 1974]\label{T:IdProjDLat}
Every finite distributive lattice is sharply transferable with respect to~$\cD$.
\end{theorem}

In particular, since not every finite distributive lattice is projective within~$\cD$ (see Balbes~\cite{Balb67}, also the comments following the proof of Theorem~\ref{T:RCML} in the present paper), Nation's above-cited result, that sharp transferability implies projectivity, does not relativize to the variety of all distributive lattices.

Theorem~\ref{T:IdProjDLat} will be further amplified in Corollary~\ref{C:IdProjDLat}.

Denote by~$\Mnoeth$ the class of all modular lattices without infinite bounded ascending chains.
The following result is established, in the course of characterizing the nonstable K-theory of certain C*-algebras, in Pardo and Wehrung \cite[Theorem~4.3$'$]{ParWeh06}.
Although this result easily implies the corresponding sharp transferability result, we do not state it in those terms, for reasons that we will explain shortly.

For a function~$\gf$ with domain~$X$, a \emph{choice function} for~$\gf$ is a function~$f$ with domain~$X$ such that $f(x)\in\gf(x)$ for any $x\in X$.

\begin{theorem}[Pardo and Wehrung, 2006]\label{T:ModNoeth}
Let~$D$ be a finite distributive lattice, let $M\in\Mnoeth$, and let $\gf\colon D\to\Id M$ be a lattice homomorphism.
Then for any choice function~$f_0$ for~$\gf$, there is a choice function~$f$ for~$\gf$, which is also a lattice homomorphism, and such that $f_0\leq f$.
\end{theorem}

Theorem~\ref{T:ModNoeth} will be further amplified in Corollary~\ref{C:Transf2Proj2}.

In the context of Theorem~\ref{T:ModNoeth}, we say that~$D$ is \emph{\idproj\ with respect to~$\Mnoeth$} (cf. Definition~\ref{D:IdProj}).
Although \idprojy\ easily implies sharp transferability (cf. the argument of the proof of Theorem~\ref{T:Transf2Proj}\eqref{CIdProj2Transf}), the converse does not hold as a rule.

For more work on projectivity with respect to the variety~$\cM$ of all modular lattices, see Day \cite{Day70,Day73,Day75}, Mitschke and Wille~\cite{MitWil76}, Freese \cite{Freese76,Fre79}.

\subsection{Description of the results}\label{Su:DescrRes}
The existing literature suggests a close connection between sharp transferability and projectivity.
One direction is actually achieved in~\cite{BakHal74}, where Baker and Hales establish that the ideal lattice~$\Id L$ of a lattice~$L$ is a homomorphic image of a sublattice of an ultrapower of~$L$.
Hence, projectivity is stronger than sharp transferability \emph{a priori}.
We push this observation further in Theorem~\ref{T:Transf2Proj}, where we relate sharp transferability to the new concept of \emph{\idprojy}, and also to projectivity with respect to the class~$\WD$ of all \emph{weakly distributive} lattice homomorphisms (cf. Definition~\ref{D:WD}).
In particular, for a variety~$\cV$ of lattices, the concepts of sharp transferability, \idprojy, and $\WD$-projectivity, all with respect to~$\cV$, are equivalent on finite members of~$\cV$ (Corollary~\ref{C:Transf2Proj1}).
This result does not extend to more general classes of lattices:
indeed, we prove in Section~\ref{S:RelCpl} that \emph{even on finite distributive lattices, sharp transferability, with respect to the class of all relatively complemented modular lattices, is \pup{properly} weaker than \idprojy\ with respect to that class}.

In order to accommodate such results as those of Huhn~\cite{Huhn72} or Freese~\cite{Freese76} (cf. Theorem~\ref{T:DiamProj}), we shall state Theorem~\ref{T:Transf2Proj} not only for lattices, but for \emph{partial lattices}.

We also supplement the above-cited Theorems~\ref{T:IdProjDLat} and~\ref{T:ModNoeth} by the following new result of sharp transferability for finite distributive lattices:
\emph{Every finite distributive lattice is sharply transferable with respect to the class of all relatively complemented modular lattices} (cf. Theorem~\ref{T:RCML}).

Most of our subsequent results will be focused on evaluating the amount of sharp transferability satisfied by the smallest lattice~$\DX$ not satisfying Whitman's Condition, represented, along with its auxiliary lattice~$\hDX$, on Figure~\ref{Fig:DandhatD}.

\begin{figure}[htb]
\includegraphics{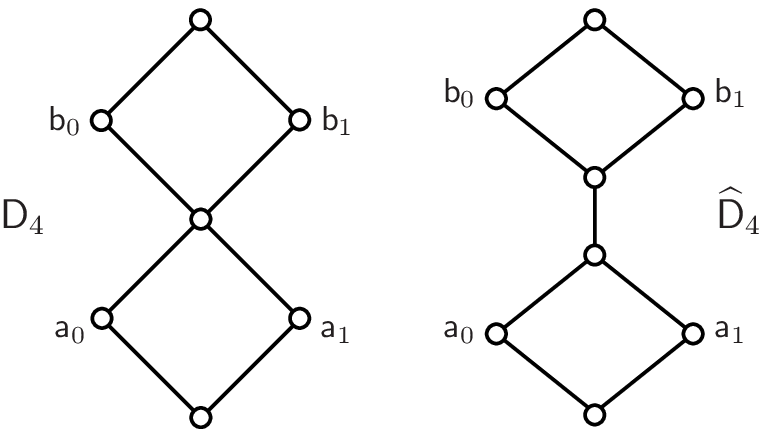}
\caption{The lattices~$\DX$ and~$\hDX$}
\label{Fig:DandhatD}
\end{figure}

We prove in Theorem~\ref{T:IdProjD} that~$\DX$ is sharply transferable, with respect to a variety~$\cV$, if{f}~$\cV$ is contained in the variety~$\cM_{\go}$ generated by all lattices of length two (which is generated by the lattice~$\sM_{\go}$ represented in Figure~\ref{Fig:MgoandM33}).
The proof of the failure of sharp transferability of~$\DX$ with respect to the variety~$\cN_5$, generated by the pentagon lattice~$\sN_5$ (cf. Figure~\ref{Fig:MgoandM33}), involves solving equations, with one unknown, in a certain class of lattices which are all finitely presented within~$\cN_5$ (cf. Lemma~\ref{L:NoMidFm}).
The proof of the corresponding result for the variety~$\cM_{3,3}$\,, generated by the lattice~$\sM_{3,3}$ represented in Figure~\ref{Fig:MgoandM33}, relies on an \emph{ad hoc} trick.

Gr\"atzer raised in \cite[Problems~I.26 and~I.27]{GLT1} the question whether transferability, with respect to a ``reasonable'' class~$\cC$ of lattices, is equivalent to sharp transferability with respect to that class.
At the time the question was asked, it was already known that it had a negative answer for infinite lattices (cf. Tan~\cite{Tan75}).
It got subsequently a positive answer for~$\cC$ defined as the variety of all lattices, and finite lattices, in Platt~\cite{Platt81}.
We prove here that~$\DX$ is transferable with respect to the variety of all modular lattices (cf. Proposition~\ref{P:D4transf}).
Since~$\DX$ is not sharply transferable with respect to that variety, this yields a negative answer to Gr\"atzer's question, even for finite distributive lattices and for~$\cC$ a variety (here, any variety~$\cC$ with $\cM_{\go}\subseteq\cC\subseteq\cM$ yields that negative answer).

Building on the proof of Theorem~\ref{T:IdProjD}, we also find, in Theorem~\ref{T:Notpure}, a modular lattice~$M$ such that the canonical embedding from~$M$ into~$\Id M$ is not pure.
This settles, in the negative, a problem raised by E. Nelson in~\cite{Nels74}.
Moreover, $M$ belongs to the variety~$\cM_{3,3}$ generated by the lattice~$\sM_{3,3}$ represented in Figure~\ref{Fig:MgoandM33}.
In particular, $M$ is $2$-distributive in Huhn's sense (cf. Huhn~\cite{Huhn72}).

\section{Notation and terminology}\label{S:Basic}

We denote by~$\fine{X}$ the set of all nonempty finite subsets of any set~$X$, and by~$\Pow{X}$ the powerset of~$X$.
We also write $\go=\set{0,1,2,\dots}$.

For unexplained concepts of lattice theory, we refer the reader to Gr\"atzer~\cite{LTF}.
For any subsets~$X$ and~$Y$ in a poset (i.e., partially ordered set)~$P$, we set
 \begin{align*}
 X\dnw Y&=\setm{x\in X}{x\leq y\text{ for some }y\in Y}\,,\\
 X\upw Y&=\setm{x\in X}{x\geq y\text{ for some }y\in Y}\,. 
 \end{align*}
We also write $X\dnw a$ (resp., $X\upw a$) instead of~$X\dnw\set{a}$ (resp., $X\upw\set{a}$), for $a\in P$.
For posets~$P$ and~$Q$, a map $f\colon P\to Q$ is \emph{isotone} (resp., \emph{antitone}) if $x\leq y$ implies that $f(x)\leq f(y)$ (resp., $f(y)\leq f(x)$) for all $x,y\in P$.
We denote by~$\Ji L$ the set of all \jirr\ elements in a lattice~$L$.
A lattice~$L$ satisfies \emph{Whitman's Condition}, or, to make it short, (W), if $y_0\wedge y_1\leq x_0\vee x_1$ implies that either $y_i\leq x_0\vee x_1$ or $y_0\wedge y_1\leq x_i$ for some $i\in\set{0,1}$.
We say that~$L$ is \emph{semidistributive} if $x\vee z=y\vee z$ implies that $x\vee z=(x\wedge y)\vee z$, for all $x,y,z\in L$, and dually.

A nonempty subset~$I$ in a \js~$S$ is an \emph{ideal} of~$S$ if for all $x,y\in S$, $x\vee y\in I$ if{f} $\set{x,y}\subseteq I$.
We denote by~$\Id S$ the set of all ideals of~$S$, partially ordered by set inclusion.
This poset is a lattice if{f}~$S$ is downward directed, in which case~$\Id S$ is a complete algebraic lattice.

A class of lattices is a \emph{variety} if it is the class of all lattices that satisfy a given set of lattice identities.
It is known since Sachs~\cite{Sachs61} that every lattice variety is closed under the assignment $L\mapsto\Id L$ (see also Gr\"atzer \cite[Lemma~59]{LTF}).

The lattices~$\sM_{\go}$\,, $\sM_{3,3}$\,, and~$\sN_5$ are represented in Figure~\ref{Fig:MgoandM33}.
They generate the varieties~$\cM_{\go}$\,, $\cM_{3,3}$\,, and~$\cN_5$\,, respectively.
The labelings introduced on Figure~\ref{Fig:MgoandM33} will be put to use in Section~\ref{S:DMgo}.
We also denote by~$\cD$ the variety of all distributive lattices, by~$\cM$ the variety of all modular lattices, and by~ $\cL$ the variety of all lattices.

It is proved in J\'onsson~\cite{Jons68} (see also Jipsen and Rose \cite[\S~3.2]{JiRo92}) that~$\cM_{\go}$ is the class of all modular lattices satisfying the following identity (lattice inclusion):
 \[
 \vx\wedge(\vy\vee(\vu\wedge\vv))\wedge(\vu\vee\vv)\leq
 \vy\vee(\vx\wedge\vu)\vee(\vx\wedge\vv)\,.
 \]
In that paper, it is also proved that a variety~$\cV$ of modular lattices is contained in~$\cM_{\go}$ if{f}~$\sM_{3,3}$ does not belong to~$\cV$.

\begin{figure}[htb]
\includegraphics{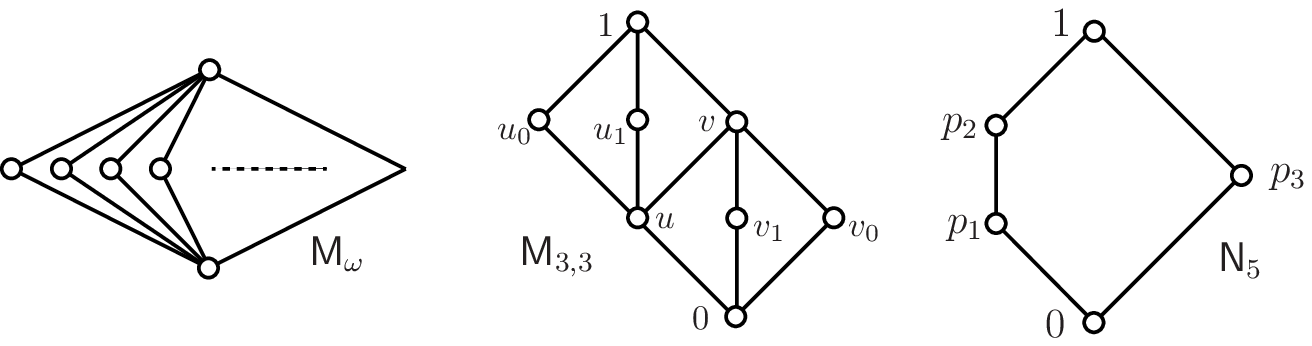}
\caption{The lattices~$\sM_{\go}$\,, $\sM_{3,3}$\,, and~$\sN_5$}
\label{Fig:MgoandM33}
\end{figure}

The following definition originates in Dean~\cite{Dean64}.

\begin{definition}\label{D:PartLatt}
A \emph{partial lattice} is a structure $(P,\leq,\bigvee,\bigwedge)$, where
$(P,\leq)$ is a poset and $\bigvee$, $\bigwedge$ are partial functions from
$\fine P$ to $P$ satisfying the following properties:
\begin{enumerate}
\item for all $a\in P$ and all $X\in\fine P$\,, the relation $a=\bigvee X$ implies that $P\upw a=\bigcap_{x\in X}(P\upw x)$;

\item for all $a\in P$ and all $X\in\fine P$\,, the relation $a=\bigwedge X$ implies that $P\dnw a=\bigcap_{x\in X}(P\dnw x)$.\end{enumerate}

If $P$ and $Q$ are partial lattices, a \emph{homomorphism of partial lattices} from~$P$ to~$Q$ is an isotone map
$f\colon P\to Q$ such that $a=\bigvee X$
(resp., $a=\bigwedge X$) implies that $f(a)=\bigvee f[X]$ (resp., $f(a)=\bigwedge f[X]$),
for all~$a\in P$ and all~$X\in\fine P$. We say that a homomorphism~$f$ of partial lattices is an
\emph{embedding} if it is an order-embedding, that is, $f(a)\leq f(b)$ implies that $a\leq b$, for all
$a,b\in P$.
\end{definition}

We shall naturally identify lattices with partial lattices $P$ such that the operations~$\bigvee$ and~$\bigwedge$ are both defined everywhere on $\fine P$.

\begin{note}
For an embedding $f\colon P\to Q$ of partial lattices, we do \emph{not}
require that $\bigvee f[X]$ be defined implies that $\bigvee X$ is defined (and
dually), for $X\in\fine P$.
\end{note}

\section{Weak distributivity, \idprojy, and sharp transferability}\label{S:IdProj}

The concept of a weakly distributive map originates in Schmidt~\cite{Schm68}.
Schmidt's definition got extended, and slightly reformulated, in various papers on the Congruence Lattice Problem (cf. Wehrung \cite{STA1-7,STA1-9}), see, in particular, \cite[Definition 7-3.4]{STA1-7}.
The reformulated definition is the following.

\begin{definition}\label{D:WD}
Let~$K$ and~$L$ be \js{s}.
A \jh\ $f\colon K\to\nobreak L$ is \emph{weakly distributive} if for all $x\in K$ and all $y_0,y_1\in L$, if $f(x)\leq y_0\vee y_1$, then there are $x_0,x_1\in K$ such that $x\leq x_0\vee x_1$ and each $f(x_i)\leq y_i$.
\end{definition}

If~$K$ and~$L$ are both downward directed (this is, for example, the case if they are both lattices), then~$f$ is weakly distributive if{f} the inverse map $f^{-1}\colon\Id L\to\Id K$ is a lattice homomorphism (cf. Wehrung \cite[Exercise~7.7]{STA1-7}).

We denote by~$\WD$ (resp., $\WDj$) the class of all surjective weakly distributive lattice homomorphisms (resp., surjective weakly distributive \jh{s} between lattices).
The following relativizes the classical definition of projectivity.

\begin{definition}\label{D:Proj}
Let~$\cC$ be a class of lattices and let~$\xF$ be a class of \emph{surjective} lattice homomorphisms.
A partial lattice~$P$ is \emph{$\xF$-projective with respect to~$\cC$} if for every $K,L\in\cC$, every lattice homomorphism $h\colon L\twoheadrightarrow K$ in~$\xF$, and every homomorphism $f\colon P\to K$ of partial lattices, there exists a homomorphism $g\colon P\to L$ of partial lattices such that $f=h\circ g$.

If~$\xF$ is the class of all surjective lattice homomorphisms, then we will just use ``projective'' instead of ``$\xF$-projective''.
\end{definition}

Observe that we do not require $P\in\cC$ in Definition~\ref{D:Proj}, even if~$P$ is a lattice.

In this paper, our classes of lattices will usually be \emph{abstract classes}, that is, closed under isomorphic copies.
Moreover, the class~$\xF$ will often be either the class of all surjective lattice homomorphisms within a given variety, or the class of all weakly distributive lattice homomorphisms within that variety.
The relevant feature common to those two classes is given by the following definition.

\begin{figure}[htb]
 \[
 \xymatrix{
 Q\ar[r]_{\ol{f}}\ar[d]_{\ol{h}} & L\ar[d]^{h}\\
 P\ar[r]^{f} & K
 }
 \]
\caption{A pullback of lattices}
\label{Fig:Pullback}
\end{figure}

\begin{definition}\label{D:PullTransf}
An abstract class~$\xF$ of lattice homomorphisms \emph{satisfies pullback transfer} if for all lattices~$P$, $K$, $L$ and all lattice homomorphisms $f\colon P\to K$ and $h\colon L\to K$, with pullback $Q=P\mathbin{\Pi}_{f,h}L$ and canonical maps $\ol{f}\colon Q\to L$ and $\ol{h}\colon Q\to P$ (cf. Figure~\ref{Fig:Pullback}), $h\in\xF$ implies that $\ol{h}\in\xF$.
\end{definition}

\begin{lemma}\label{L:WD2PullTransf}
The class of all surjective lattice homomorphisms, and the class~$\WD$ of all surjective weakly distributive lattice homomorphisms, both satisfy pullback transfer.
\end{lemma}

\begin{proof}
Given $P$, $K$, $L$, $f$, $h$ as on Figure~\ref{Fig:Pullback}, it is well known (and easy to verify) that up to isomorphism, the pullback is given by $Q=\setm{(x,y)\in P\times L}{f(x)=h(y)}$, with $\ol{f}(x,y)=\nobreak y$ and $\ol{h}(x,y)=x$ whenever $(x,y)\in Q$.
It is trivial that the surjectivity of~$h$ implies the one of~$\ol{h}$.

Now suppose that~$h\in\WD$ and let $(x,y)\in Q$ and $x_0,x_1\in P$ such that $\ol{h}(x,y)\leq x_0\vee x_1$; that is, $x\leq x_0\vee x_1$.
It follows that $h(y)=f(x)\leq f(x_0)\vee f(x_1)$.
Since~$h$ is weakly distributive, there are $y_0,y_1\in L$ such that $y\leq y_0\vee y_1$ and each $h(y_i)\leq f(x_i)$.
Since~$h$ is surjective, we may enlarge each~$y_i$ in such a way that each $h(y_i)=f(x_i)$.
It follows that $(x,y)\leq(x_0,y_0)\vee(x_1,y_1)$ with each $(x_i,y_i)\in Q$ and $\ol{h}(x_i,y_i)=x_i$\,, thus completing the verification of pullback transfer for~$\ol{h}$.
\end{proof}

The following lemma enables us to simplify the definition of projectivity in the presence of pullback transfer and under mild conditions on~$\cC$.

\begin{lemma}\label{D:PullTransfProj}
Let~$\cC$ be an abstract class of lattices, closed under sublattices and nonempty finite products, and let~$\xF$ be a class of surjective lattice homomorphisms satisfying pullback transfer.
Then a lattice~$P$ is $\xF$-projective with respect to~$\cC$ if{f} for every $Q\in\cC$, every map $h\colon Q\twoheadrightarrow P$ in~$\xF$ is a lattice retraction \pup{i.e., there is a lattice embedding $f\colon P\hookrightarrow Q$ such that $h\circ f=\id_P$}.
\end{lemma}

\begin{proof}
If~$P$ is $\xF$-projective with respect to~$\cC$, then it trivially satisfies the given condition (consider the identity map on~$P$).

Suppose, conversely, that the given condition holds.
Let $K,L\in\cC$, with lattice homomorphisms $h\colon L\twoheadrightarrow K$ in~$\xF$ and $f\colon P\to K$.
Consider the pullback $Q=P\mathbin{\Pi}_{f,h}L$ (cf. Figure~\ref{Fig:Pullback}).
Since~$Q$ is a sublattice of $P\times L$ and by assumption on~$\cC$, $Q$ belongs to~$\cC$.
Furthermore, by assumption on~$\xF$, $\ol{h}$ belongs to~$\xF$.
Now our assumption implies the existence of a lattice embedding $g\colon P\hookrightarrow Q$ such that $\ol{h}\circ g=\id_P$.
Observe that $h\circ\ol{f}\circ g=f\circ\ol{h}\circ g=f$, thus completing the verification of the $\xF$-projectivity of~$P$ with respect to~$\cC$.
\end{proof}

\begin{definition}\label{D:ChoiceFct}
Let~$\gf$ be a set-valued function with domain a set~$X$.
A \emph{choice function} for~$\gf$ is a function~$f$ with domain~$X$ such that $f(x)\in\gf(x)$ for every $x\in X$.
If $\gf\colon X\to\Pow Y$, $X$ and~$Y$ are objects in a (concrete) category~$\cC$, and~$f$ is a morphism in~$\cC$, we then say that~$f$ is a \emph{choice morphism} for~$\gf$.
We extend this terminology to special classes of morphisms, such as choice homomorphisms, choice embeddings, and so on.
\end{definition}

\begin{definition}\label{D:IdProj}
Let~$\cC$ be a class of lattices.
A partial lattice~$P$ is
\begin{itemize}
\item
\emph{\idproj\ with respect to~$\cC$} if for every $L\in\cC$, every homomorphism $\gf\colon P\to\Id L$ of partial lattices, and every choice function~$f_0$ for~$\gf$, there exists a choice homomorphism of partial lattices $f\colon P\to L$ for~$\gf$ such that $f_0\leq f$.

\item
\emph{sharply transferable with respect to~$\cC$} if for every $L\in\cC$ and every embedding $\gf\colon P\hookrightarrow\Id L$ of partial lattices, there exists a choice homomorphism of partial lattices $f\colon P\to L$ for~$\gf$ such that $f(x)\in\gf(y)$ implies $x\leq y$ whenever $x,y\in P$ --- we then say that~$f$ satisfies the \emph{transfer condition with respect to~$\gf$}.

\item
\emph{transferable with respect to~$\cC$} if for every $L\in\cC$, if~$P$ embeds into~$\Id L$ as a partial lattice, then it also embeds into~$L$ as a partial lattice.
\end{itemize}
\end{definition}

Observe that in the context of Definition~\ref{D:IdProj}, if a choice homomorphism $f\colon P\to\nobreak L$ for~$\gf$ satisfies the transfer condition with respect to~$\gf$, then it is an embedding, and every choice homomorphism~$g$ for~$\gf$, such that $f\leq g$, also satisfies the transfer condition with respect to~$\gf$.
Also, transferability and sharp transferability of~$P$ with respect to~$\cC$ make sense only in case~$P$ embeds into the ideal lattice of some lattice in~$\cC$ (otherwise, both transferability and sharp transferability are vacuously satisfied).

The following result is essentially contained in Baker and Hales \cite[Theorem~A]{BakHal74}.
We include a proof, with an amendment regarding weak distributivity, for convenience.

\begin{proposition}[Baker and Hales, 1974]\label{P:BaHaRepr}
For every lattice~$L$, the ideal lattice~$\Id L$ of~$L$ is the image, under a weakly distributive lattice homomorphism, of a sublattice of an ultrapower of~$L$.
\end{proposition}

\begin{proof}
The set~$\gL$ of all finite subsets of~$L$, partially ordered by set inclusion, is \emph{lower finite}, that is, the subset $\dnw\gl=\setm{\xi\in\gL}{\xi\subseteq\gl}$ is finite for every $\gl\in\gL$.
Let~$\cU$ be an ultrafilter on~$\gL$ such that each subset $\upw\gl=\setm{\xi\in\gL}{\gl\subseteq\xi}$, for $\gl\in\gL$, belongs to~$\cU$.
Denote by~$L^{\gL}/{\cU}$ the ultrapower of~$L$ by~$\cU$, and by $\rho\colon L^{\gL}\twoheadrightarrow L^{\gL}/{\cU}$ the canonical projection.

The set~$S$ of all isotone maps from~$\gL$ to~$L$ is a sublattice of~$L^{\gL}$, thus the set $T=\rho[S]$ is a sublattice of $L^{\gL}/{\cU}$.
Since all elements of~$\cU$ are cofinal subsets of~$\gL$, we obtain that for every $x\in S$, all the subsets~$x[U]$, where $U\in\cU$, generate the same ideal of~$L$, which we can thus denote by~$\pi(x/{\cU})$ (where~$x/{\cU}$ denotes the equivalence class of~$x$ modulo~$\cU$).
The map~$\pi$ is a lattice homomorphism from~$T$ to~$\Id L$.
The situation is illustrated on Figure~\ref{Fig:BakHal}.

Let~$\bx\in\Id L$ and pick $o\in\bx$.
The assignment $\gl\mapsto o\vee\bigvee(\gl\cap\bx)$ (where we define $o\vee\bigvee\es=o$) defines an isotone map $x\colon\gL\to L$, and $\pi(x/{\cU})=\bx$.
Hence, $\pi$ is surjective.

We claim that~$\pi$ is weakly distributive.
Let $x\in S$ and let $\ba,\bb\in\Id L$ such that $\pi(x/{\cU})\leq\ba\vee\bb$.
The latter inequality means that $x_{\gl}\in\ba\vee\bb$ for every $\gl\in\gL$.
By using the lower finiteness of~$\gL$, we can construct inductively~$a_{\gl}$ and~$b_{\gl}$, for $\gl\in\gL$, as follows.
Suppose that $a_{\xi}\in\ba$ and $b_{\xi}\in\bb$ for every $\xi<\gl$.
Since $x_{\gl}\in\ba\vee\bb$, there are $a_{\gl}\in\ba$ and $b_{\gl}\in\bb$ such that $x_{\gl}\leq a_{\gl}\vee b_{\gl}$.
By joining~$a_{\gl}$ with $\bigvee_{\xi<\gl}a_{\xi}$ and~$b_{\gl}$ with $\bigvee_{\xi<\gl}b_{\xi}$\,, we may assume that $a_{\xi}\leq a_{\gl}$ and $b_{\xi}\leq b_{\gl}$ for every $\xi<\gl$.
This way, the maps~$a$ and~$b$ both belong to~$S$, and $x\leq a\vee b$ (thus $x/{\cU}\leq a/{\cU}\vee b/{\cU}$), with $\pi(a/{\cU})\subseteq\ba$ and $\pi(b/{\cU})\subseteq\bb$.
This completes the proof of our claim.
\end{proof}

The following result investigates the connections between $\WD$-projectivity and (sharp) transferability relative to a class~$\cC$ of lattices, giving equivalences under certain conditions on~$\cC$.

\goodbreak

\begin{theorem}\label{T:Transf2Proj}
Let~$\cC$ be an abstract class of lattices and let~$P$ be a partial lattice.
Consider the following statements:
\begin{enumerater}
\item\label{PWDIdProj}
$P$ is \idproj\ with respect to~$\cC$.

\item\label{PWDjProjC}
$P$ is $\WDj$-projective with respect to~$\cC$.

\item\label{PWDProjC}
$P$ is $\WD$-projective with respect to~$\cC$.

\item\label{PtransfC}
$P$ is sharply transferable with respect to~$\cC$.

\item\label{PwtransfC}
$P$ is transferable with respect to~$\cC$.

\end{enumerater}

Then the following statements hold:
\begin{enumeratea}
\item\label{CIfProj2Proj}
The implications \eqref{PWDIdProj}$\Rightarrow$\eqref{PWDjProjC}$\Rightarrow$\eqref{PWDProjC} and \eqref{PtransfC}$\Rightarrow$\eqref{PwtransfC} always hold.

\item\label{CProj2IdProj}
If~$P$ is finite and~$\cC$ is closed under sublattices, images under weakly distributive lattice homomorphisms, and ultraproducts, then~\eqref{PWDProjC} implies~\eqref{PWDIdProj}.

\item\label{CIdProj2Transf}
If~$P$ is finite, then~\eqref{PWDIdProj} implies~\eqref{PtransfC}.

\item\label{CTransf2Proj}
If $P\in\cC$ and~$\cC$ is closed under sublattices and nonempty finite products, then~\eqref{PtransfC} implies~\eqref{PWDProjC}.
\end{enumeratea}

\end{theorem}

\begin{figure}[htb]
 \[
 \xymatrix{
 S\,\ar@{_(->}[r]\ar@{->>}[d]^{\rho\res_S} & L^{\gL}\ar@{->>}[d]^{\rho} & & &
 S\,\ar@{_(->}[r]\ar@{->>}[d]^{\rho\res_S} & L^{\gL}\ar@{->>}[d]^{\rho}\\
 T\,\ar@{->>}[d]^{\pi}\ar@{^(->}[r] & L^{\gL}/{\cU} & & & 
 T\,\ar@{->>}[d]^{\pi}\ar@{^(->}[r] & L^{\gL}/{\cU}\\
 \Id L & & & P\ar[ru]_{\!\!\psi}\ar[r]_{\varphi}\ar[ruu]^{\dot{\psi}} & \Id L &
 }
 \]
\caption{Illustrating the proofs of Proposition~\ref{P:BaHaRepr} (left)
and Theorem~\ref{T:Transf2Proj}\eqref{CProj2IdProj} (right)}
\label{Fig:BakHal}
\end{figure}
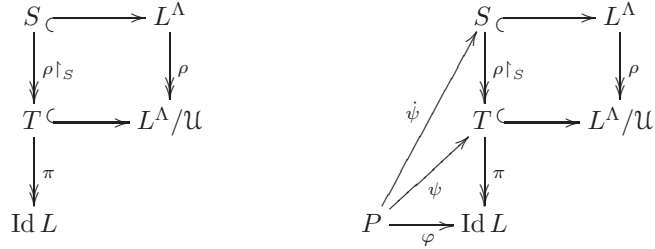

\begin{proof}
\eqref{CIfProj2Proj}.
The only nontrivial implication to be verified is \eqref{PWDIdProj}$\Rightarrow$\eqref{PWDjProjC}.
Let~$P$ be \idproj\ with respect to~$\cC$.
Let $K,L\in\cC$, let $f\colon P\to K$ be a homomorphism of partial lattices, and let $h\colon L\twoheadrightarrow K$ be a surjective weakly distributive \jh.
We set $\gf(x)=\setm{t\in L}{h(t)\leq f(x)}$, for any $x\in P$.
Since~$h$ is weakly distributive, $\gf$ is a homomorphism of partial lattices from~$P$ to~$\Id L$.
Since~$h$ is surjective, for each $x\in P$, there exists $g_0(x)\in L$ such that $h(g_0(x))=f(x)$.
In particular, $g_0$ is a choice function for~$\gf$.
By assumption, there is a choice homomorphism of partial lattices $g\colon P\to L$ for~$\gf$ such that $g_0\leq g$.
Let $x\in P$.
On the one hand, $g(x)\in\gf(x)$ means that $(h\circ g)(x)\leq f(x)$.
On the other hand, $g_0(x)\leq g(x)$ implies that $f(x)=(h\circ g_0)(x)\leq(h\circ g)(x)$.
Therefore, $h\circ g=f$.

\eqref{CProj2IdProj}.
Suppose that $P$ is $\WD$-projective with respect to~$\cC$, let~$L\in\cC$, and let $\gf\colon P\to\nobreak\Id L$ be a lattice homomorphism.
We apply Baker and Hales' representation of~$\Id L$ obtained from Proposition~\ref{P:BaHaRepr}.
We keep the notation of the proof of that lemma.
The situation is illustrated on Figure~\ref{Fig:BakHal}.
It follows from Proposition~\ref{P:BaHaRepr}, together with the assumptions on~$\cC$, that~$T$ and~$\Id L$ both belong to~$\cC$.

Since~$P$ is $\WD$-projective with respect to~$\cC$, there exists a lattice homomorphism $\psi\colon P\to\nobreak T$ such that $\gf=\pi\circ\psi$.
Since~$\rho$ is surjective, there is a map $\dot{\psi}\colon P\to S$ (not a homomorphism \emph{a priori}) such that $\psi=\rho\circ\dot{\psi}$.
Set $\dot{\psi}(x)=\vecm{\dot{\psi}_{\gl}(x)}{\gl\in\gL}$, for every $x\in P$.
Since~$\psi$ is a lattice homomorphism, all sets of the form
 \begin{align*}
 U^{\vee}_X&=\setm{\gl\in\gL}
 {\dot{\psi}_{\gl}\pI{\bigvee X}=\bigvee\dot{\psi}_{\gl}[X]}\,,
 &&\text{for }X\text{ in the domain of }\bigvee\,,\\
 U^{\wedge}_X&=\setm{\gl\in\gL}
 {\dot{\psi}_{\gl}\pI{\bigwedge X}=\bigwedge\dot{\psi}_{\gl}[X]}\,,
 &&\text{for }X\text{ in the domain of }\bigwedge\,,
 \end{align*}
belong to~$\cU$.
Since~$P$ is finite, the intersection~$U$, of all sets of the form~$U^{\vee}_X$ or~$U^{\wedge}_X$, belongs to~$\cU$.
Observe that
 \begin{equation}\label{Eq:dotpsialmhom}
 \dot{\psi}_{\gl}\text{ is a homomorphism of partial lattices from }
 P\text{ to }L\,,\text{ for every }\gl\in U\,.
 \end{equation}
Now let~$f_0\colon P\to L$ be a choice function for~$\gf$.
For each $x\in P$, $f_0(x)$ belongs to $\gf(x)=\pi(\psi(x))=\pi\pI{\dot{\psi}(x)/{\cU}}$, which is the ideal generated by the range of the map~$\dot{\psi}(x)$.
Hence, $f_0(x)\leq\dot{\psi}_{\gl}(x)$ for all large enough $\gl\in\gL$.
Since~$P$ is finite and~$U$ is cofinal in~$\gL$, there exists $\gl\in U$ such that $f_0(x)\leq\dot{\psi}_{\gl}(x)$ for every $x\in P$.
The map $f=\dot{\psi}_{\gl}$ is a choice function for~$\gf$, and $f_0\leq f$.
Furthermore, by~\eqref{Eq:dotpsialmhom}, $f$ is a lattice homomorphism.

\eqref{CIdProj2Transf}.
We suppose that~\eqref{PWDIdProj} holds.
Let $L\in\cC$ and let $\gf\colon P\hookrightarrow\Id L$ be an embedding of partial lattices.
Pick $o\in\gf(0_P)$, and pick $a_{x,y}\in\gf(x)\setminus\gf(y)$, for all $x,y\in P$ such that $x\nleq y$.
For every $x\in P$, it follows from the finiteness of~$P$ that we can define $f_0(x)=o\vee\bigvee\vecm{a_{x,y}}{y\in P\,,\ x\nleq y}$ (the right hand side of that expression being defined as being equal to~$o$ in case the big join is empty).
Since $o\in\gf(0_P)$ and each $a_{x,y}\in\gf(x)$, $f_0$ is a choice function for~$\gf$.
By our assumption, there is a choice homomorphism $f\colon P\to L$ for~$\gf$ such that $f_0\leq f$.
We need to prove that $x\nleq y$ implies that $f(x)\notin\gf(y)$, for all $x,y\in P$.
This holds indeed, because $a_{x,y}\leq f_0(x)\leq f(x)$ while $a_{x,y}\notin\gf(y)$.

\eqref{CTransf2Proj}.
By applying Lemmas~\ref{L:WD2PullTransf} and~\ref{D:PullTransfProj} to $\xF=\WD$, we see that it suffices to prove that for every $P\in\cC$ which is sharply transferable with respect to~$\cC$, and for every surjective weakly distributive lattice homomorphism $h\colon Q\twoheadrightarrow P$, there is a lattice embedding $f\colon P\hookrightarrow Q$ such that $h\circ f=\id_P$.
In fact, we shall prove a little more: indeed, $h$ will only need to be a surjective weakly distributive \emph{\jh}.

Define a map $\gf\colon P\to\nobreak\Id Q$ by setting $\gf(x)=\setm{t\in Q}{h(t)\leq x}$, for every $x\in P$.
It is trivial that~$\gf$ is a \mh.
Furthermore, since~$h$ is weakly distributive, $\gf$ is also a \jh; whence it is a lattice homomorphism.
Since~$h$ is surjective, $x$ is the largest element of $h[\gf(x)]$, for every $x\in P$; whence~$\gf$ is a lattice embedding.
Since~$P$ is sharply transferable with respect to~$\cC$, there is a lattice embedding $f\colon P\hookrightarrow Q$ such that
 \[
 f(x)\in\gf(y)\ \Longleftrightarrow\ x\leq y\,,\quad
 \text{for all }x,y\in P\,.
 \]
This means that $(h\circ f)(x)\leq y$ if{f} $x\leq y$, for all $x,y\in P$; whence $h\circ f=\id_P$.
\end{proof}

\begin{corollary}\label{C:Transf2Proj1}
Let~$\cC$ be an abstract class of lattices, closed under sublattices, nonempty finite products, images under weakly distributive lattice homomorphisms, and ultrapowers, and let~$P$ be a finite member of~$\cC$.
Then the following are equivalent:
\begin{enumeratea}
\item\label{CPWDProjC}
$P$ is $\WDj$-projective with respect to~$\cC$.

\item\label{CjPWDProjC}
$P$ is $\WD$-projective with respect to~$\cC$.

\item\label{CPWDIdProj}
$P$ is \idproj\ with respect to~$\cC$.

\item\label{CPtransfC}
$P$ is sharply transferable with respect to~$\cC$.
\end{enumeratea}
\end{corollary}

In particular, Corollary~\ref{C:Transf2Proj1} applies to the case where~$\cC$ is a variety of lattices.
It also applies to the case where~$\cC$ is the class of all lattices of finite length within a given variety.
Nevertheless, we will see in Section~\ref{S:RelCpl} that Corollary~\ref{C:Transf2Proj1} does not extend to more general, although natural, classes of lattices.
We will also see that even for finite distributive lattices, transferability and sharp transferability, with respect to a given lattice variety, are distinct concepts (cf. Proposition~\ref{P:D4transf}).

A direct application of Theorems~\ref{T:Transf2Proj} and~\ref{T:IdProjDLat} yields the following.

\begin{corollary}\label{C:IdProjDLat}
Every finite distributive lattice is $\WDj$-projective, \idproj, and sharply transferable, all with respect to~$\cD$.
\end{corollary}

In particular, for every finite distributive lattice~$D$, every distributive lattice~$E$, and every surjective weakly distributive \jh\ $h\colon E\twoheadrightarrow D$, there is a lattice embedding $f\colon D\hookrightarrow E$ such that $h\circ f=\id_D$.

Similarly, a direct application of Theorems~\ref{T:Transf2Proj} and~\ref{T:ModNoeth} yields the following.

\begin{corollary}\label{C:Transf2Proj2}
Every finite distributive lattice is $\WDj$-projective, \idproj, and sharply transferable, all with respect to~$\Mnoeth$.
\end{corollary}

In particular, for every finite distributive lattice~$D$, every $M\in\Mnoeth$, and every surjective weakly distributive \jh\ $h\colon M\twoheadrightarrow D$, there is a lattice embedding $f\colon D\hookrightarrow M$ such that $h\circ f=\id_D$.

\section{Relatively complemented lattices}\label{S:RelCpl}

{}From now on we shall denote by~$\cR$ the class of all relatively complemented lattices, and by~$\cRm$ the class of all modular members of~$\cR$.
One of the consequences of the present section will be that the conclusion of Corollary~\ref{C:Transf2Proj1} does not extend to the class~$\cRm$\,.

\begin{theorem}\label{T:RCML}
Every finite distributive lattice is sharply transferable with respect to~$\cRm$\,.
\end{theorem}

\begin{proof}
Let~$D$ be a finite distributive lattice, let $M\in\cRm$\,, and let $\gf\colon D\hookrightarrow\Id M$ be a lattice embedding.
We need to find a choice homomorphism for~$\gf$ which satisfies the transfer condition with respect to~$\gf$.

Set $P=\Ji D=\set{p_1,\dots,p_m}$ in such a way that $p_i\leq p_j$ implies $i\leq j$, for all $i,j\in[1,m]$.
Moreover, fix $o\in\gf(0_D)$.
Since~$D$ is finite distributive, for any $p\in P$, there is a largest $p^{\dagger}\in D$ such that $p\nleq p^{\dagger}$.
Observe that $p\wedge p^{\dagger}=p_*$, the unique lower cover of~$p$.
Since~$\gf$ is an embedding, for any $k\in[1,m]$, there is $a_k\in\gf(p_k)\setminus\gf(p_k^{\dagger})$.
Replacing~$a_k$ by~$a_k\vee o$, we may assume that each $a_k\geq o$.
Set $a_{<k}=\bigvee_{1\leq i<k}a_i$\,, where the empty join is set equal to~$o$.
Since~$M$ is relatively complemented, the element $a_k\wedge a_{<k}$ has a relative complement~$b_k$ in the interval $[o,a_k]$, for each $k\in[1,m]$.

\setcounter{claim}{0}

\begin{claim}\label{Cl:Basicaka<k}
$a_k\wedge a_{<k}\in\gf(p_{k*})$ and $b_k\in\gf(p_k)\setminus\gf(p_k^{\dagger})$, for any $k\in[1,m]$.
\end{claim}

\begin{cproof}
Observing that $p_k\nleq p_i$ whenever $1\leq i<k$, we obtain that $p_i\leq p^{\dagger}_k$.
It follows that $a_{<k}\in\gf(p^{\dagger}_k)$, so $a_k\wedge a_{<k}\in\gf(p_k)\cap\gf(p^{\dagger}_k)=\gf(p_{k*})$.
{}From $b_k\leq a_k$ it follows that $b_k\in\gf(p_k)$.
Furthermore, from $a_k=(a_k\wedge a_{<k})\vee b_k$\,, $a_k\wedge a_{<k}\in\gf(p^{\dagger}_k)$, and $a_k\notin\gf(p^{\dagger}_k)$ it follows that $b_k\notin\gf(p^{\dagger}_k)$.
\end{cproof}

\begin{claim}\label{Cl:Indepbk}
The finite sequence $(b_1,\dots,b_m)$ is independent over~$o$.
\end{claim}

\begin{cproof}
Set $b_{<k}=\bigvee_{1\leq i<k}b_i$\,, for every $k\in[1,m]$.
Since~$M$ is modular, it suffices (cf. Gr\"atzer \cite[Theorem~360]{LTF}) to prove that $b_k\wedge b_{<k}=o$ whenever $1\leq k\leq m$.
Since $o\leq b_{<k}\leq a_{<k}$ and $a_{<k}\wedge b_k=o$, this is obvious.
\end{cproof}

By Claim~\ref{Cl:Indepbk} and since~$D$ is distributive, the map $f\colon D\to M$ defined by the rule $f(x)=\bigvee\vecm{b_i}{i\in[1,m]\,,\ p_i\leq x}$, for any $x\in D$, is a lattice homomorphism from~$D$ to~$M$.
Since each $b_i\in\gf(p_i)$, $f$ is a choice function for~$\gf$.

It remains to verify that~$f$ satisfies the transfer condition with respect to~$\gf$.
For this, it suffices in turn to verify that $f(p_k)\in\gf(x)$ implies that $p_k\leq x$, for each $(k,x)\in[1,m]\times D$.
Suppose, to the contrary, that $p_k\nleq x$.
This means that $x\leq p^{\dagger}_k$.
It follows that $f(p_k)\in\gf(p^{\dagger}_k)$, a contradiction since $b_k\leq f(p_k)$ and by Claim~\ref{Cl:Basicaka<k}.
\end{proof}

We do not know whether every finite distributive lattice is projective with respect to~$\cRm$\,.
Since no section of the canonical surjective homomorphism from~$\hDX$ onto~$\DX$ (cf. Figure~\ref{Fig:DandhatD}) is a lattice homomorphism, $\DX$ is not projective with respect to~$\cD$ (see Balbes~\cite{Balb67} for a much more general result).
Hence, $\DX$ it is also not projective with respect to~$\cM$.
However, $\cRm$ is properly contained in~$\cM$, so Balbes' result does not say anything about projectivity with respect to~$\cRm$\,, unless it already holds in~$\cM$.
And indeed, we can state the following.

\begin{theorem}\label{T:D4RelCplProj}
The lattice~$\DX$ is projective with respect to~$\cR$.
\end{theorem}

\begin{proof}
We must prove that whenever~$K$ and~$L$ are relatively complemented lattices, $h\colon L\twoheadrightarrow K$ is a surjective lattice homomorphism, and $a_0,a_1,b_0,b_1\in K$ such that $a_0\vee a_1=b_0\wedge b_1$, there are $x_0,x_1,y_0,y_1\in L$ such that each $h(x_i)=a_i$\,, each $h(y_i)=\nobreak b_j$\,, and $x_0\vee x_1=y_0\wedge y_1$.
Since~$h$ is surjective, there are $u_0,x_1,y_0,y_1\in L$ such that $h(u_0)=a_0$, $h(x_1)=a_1$, and each $h(y_i)=b_j$\,.
By replacing each~$y_i$ by $y_i\vee u_0\vee x_1$, we may assume that $u_0\vee x_1\leq y_0\wedge y_1$.
Since~$L$ is relatively complemented, the element $u_0\vee x_1$ has a relative complement~$x_0$ in the interval $[u_0,y_0\wedge y_1]$.
{}From $h(u_0\vee x_1)=h(y_0\wedge y_1)$ it follows that $h(x_0)=h(u_0)=a_0$\,.
Furthermore, by definition, $x_0\vee x_1=x_0\vee u_0\vee x_1=y_0\wedge y_1$.
\end{proof}

\begin{theorem}\label{T:D4NonIdProj}
The lattice~$\DX$ is not \idproj\ with respect to~$\cRm$\,.
\end{theorem}

\begin{proof}
For any field~$\Bbbk$, we consider distinct symbols~$a_n$ and~$b_n$, for $n<\go$, and the vector space~$E$ over~$\Bbbk$ with basis $\setm{a_n}{n<\go}\cup\setm{b_n}{n<\go}$.
For any family $\vecm{x_i}{i\in I}$ of elements of~$E$, we denote by $\seqm{x_i}{i\in I}$ the vector subspace of~$E$ generated by $\setm{x_i}{i\in I}$.
We also write $\seq{x_1,\dots,x_n}$ instead of $\seqm{x_i}{i\in[1,n]}$, and so on, for sequences enumerated by intervals of~$\go$.
We consider the complemented modular (thus relatively complemented) lattice $L=\Sub E$ of all subspaces of the vector space~$E$.

We set $A_n=\seq{a_0,\dots,a_n}$ and $B_n=\seq{b_0,\dots,b_n}$, for each $n<\go$, and further,
 \begin{align*}
 C_0&=\seq{a_0+b_0,a_1+b_1+a_0,a_2+b_2+a_1,
 a_3+b_3+a_2,\dots}\,,\\
 D_0&=\seq{a_0+b_0,a_1+b_1,a_2+b_2,a_3+b_3,\dots}\,,
 \end{align*}
and
 \begin{align*}
 C_{n+1}&=C_0+A_n+B_n\,,\\
 D_{n+1}&=D_0+A_n+B_n\,, 
 \end{align*}
for each $n<\go$.
Elementary calculations yield that
 \begin{align*}
 C_{n+1}&=\seq{a_0,b_0,\dots,a_n,b_n,a_{n+1}+b_{n+1},
 a_{n+2}+b_{n+2}+a_{n+1},a_{n+3}+b_{n+3}+a_{n+2},\dots}\,,\\
 D_{n+1}&=\seq{a_0,b_0,\dots,a_n,b_n,a_{n+1}+b_{n+1},
 a_{n+2}+b_{n+2},a_{n+3}+b_{n+3},\dots}\,, 
 \end{align*}
for all $n<\go$.
Hence, a further elementary calculation yields
 \[
 C_n\cap D_n=\seq{a_0,b_0,\dots,a_{n-1},b_{n-1},a_n+b_n}\,,
 \quad\text{for each }n<\go\,.
 \]
It follows that $C_n\cap D_n\subseteq A_n+B_n\subseteq C_{n+1}\cap D_{n+1}$, for all $n<\go$.
Hence, denoting by~$\ba_0$, $\ba_1$, $\bb_0$, $\bb_1$ the ideals of~$L$ generated by $\setm{A_n}{n<\go}$, $\setm{B_n}{n<\go}$, $\setm{C_n}{n<\go}$, $\setm{D_n}{n<\go}$, respectively, we obtain that $\ba_0\vee\ba_1=\bb_0\cap\bb_1$ in~$\Id L$.

Suppose that~$\DX$ is \idproj\ with respect to~$\cRm$\,.
Then there are $X_i\in\ba_i$ and $Y_i\in\bb_i$, for $i\in\set{0,1}$, such that $X_0+X_1=Y_0\cap Y_1$, $C_0\subseteq Y_0$, and $D_0\subseteq Y_1$.
There is $m<\go$ such that $X_0\subseteq A_m$ and $X_1\subseteq B_m$.
Setting $Z=X_0+X_1=Y_0\cap Y_1$, it follows that $X_0\subseteq Z\cap A_m$, $X_1\subseteq Z\cap B_m$, $Z+C_0\subseteq Y_0$, and $Z+D_0\subseteq Y_1$, thus
 \begin{equation}\label{Eq:ZABCDn}
 Z=(Z\cap A_m)+(Z\cap B_m)=(Z+C_0)\cap(Z+D_0)\,.
 \end{equation}
We claim that $\set{a_n,b_n}\subseteq Z$, for each $n<\go$.
We argue by induction on~$n$.
Suppose having proved that $\set{a_k,b_k}\subseteq Z$ for each $k<n$.
Since $C_0\subseteq Z+C_0$, it follows that $C_n\subseteq Z+C_0$, thus $a_n+b_n\in Z+C_0$.
A similar proof yields that $a_n+b_n\in Z+D_0$.
By~\eqref{Eq:ZABCDn}, it follows that $a_n+b_n\in(Z\cap A_m)+(Z\cap B_m)$, thus, \emph{a fortiori}, $a_n+b_n\in(Z\cap A_{m'})+(Z\cap B_{m'})$ where we set $m'=\max\set{m,n}$.
Since $a_n\in A_{m'}$\,, $b_n\in B_{m'}$\,, and $A_{m'}\cap B_{m'}=\set{0}$, we get $a_n\in Z\cap A_{m'}$ and $b_n\in Z\cap B_{m'}$, thus completing the induction step.
Now, our claim at stage $n=m+1$ yields that $a_{m+1}\in(Z\cap A_m)+(Z\cap B_m)\subseteq A_m+B_m$, a contradiction.
\end{proof}

\section[A specific example]{Varieties for which~$\DX$ is sharply transferable}\label{S:DMgo}

This section will be focused on the lattice~$\DX$ introduced in Section~\ref{S:Intro} (cf. Figure~\ref{Fig:DandhatD}).
Collapsing the central segment $[\va_0\vee\va_1,\vb_0\wedge\vb_1]$ defines a surjective lattice homomorphism from~$\hDX$ onto~$\DX$.
Since no section of that map is a lattice homomorphism, and since~$\hDX$ is distributive, $\DX$ is not projective with respect to~$\cD$.

Further, by the characterization of sharp transferability given in Gaskill, Gr\"atzer, and Platt \cite[Theorem~4.4]{GaGrPl75}, $\DX$ is not sharply transferable with respect to~$\cL$ (because it fails Whitman's Condition).
In fact, $\DX$ is the smallest lattice failing Whitman's Condition (thus failing transferability with respect to the variety~$\cL$ of all lattices).

On the other hand, it follows from Theorem~\ref{T:IdProjDLat} that~$\DX$ is sharply transferable with respect to~$\cD$.
The present section is devoted to pushing this observation a bit further.

Our first lemma introduces a quasi-identity satisfied by the variety~$\cM_{\go}$\,.

\begin{lemma}\label{L:QuasiIdMgo}
Let $M\in\cM_{\go}$ and let $a_0,a_1,b_0,b_1,a'_0,a'_1,b'_0,b'_1\in M$ satisfy the following conditions:
 \begin{align}
 a_i&\leq a'_i
 &&(\text{for all }i\in\set{0,1})\,;\label{Eq:ai0leqai1}\\
 b'_i&=b_i\vee a_0\vee a_1
 &&(\text{for all }i\in\set{0,1})\,;\label{Eq:bi02bi1}\\
 b_0\wedge b_1&\leq a_0\vee a_1\,;\label{Eq:bWa0}\\
 b'_0\wedge b'_1&\leq a'_0\vee a'_1\,.\label{Eq:bWa1}
 \end{align}
Set $a^*_i=a'_i\wedge b'_0\wedge b'_1$, for each $i\in\set{0,1}$.
Then $a^*_0\vee a^*_1=b'_0\wedge b'_1$.
\end{lemma}

\begin{proof}
By Birkhoff's Theorem, every member of~$\cM_{\go}$ is a subdirect product of subdirectly irreducible members of~$\cM_{\go}$, so it suffices to verify our statement in case~$M$ is subdirectly irreducible.
Then it follows from J\'onsson's Lemma (cf. J\'onsson \cite[Corollary~3.2]{Jons67}) that~$M$ is a homomorphic image of a lattice of length at most~$2$, thus~$M$ has length at most~$2$.

Next, we observe the following obvious consequence of~\eqref{Eq:ai0leqai1} and~\eqref{Eq:bi02bi1}:
 \begin{equation}\label{Eq:ai0leai}
 a_i\leq a^*_i\leq a'_i\,,\quad\text{for each }i\in\set{0,1}\,.
 \end{equation}
The inequality $a^*_0\vee a^*_1\leq b'_0\wedge b'_1$ is trivial.
Furthermore, if $a_0=a_1=0$, then each $b'_i=b_i$\,, thus, by~\eqref{Eq:bWa0}, $b'_0\wedge b'_1=b_0\wedge b_1
\leq a_0\vee a_1\leq a^*_0\vee a^*_1$ and we are done.

Now suppose, towards a contradiction, that $a^*_0\vee a^*_1<b'_0\wedge b'_1$.
By the paragraph above, together with~\eqref{Eq:ai0leai}, the element $a^*_0\vee a^*_1$ is nonzero.
Since~$M$ has length at most~$2$, it follows that $b'_0\wedge b'_1=1$, whence each $a^*_i=a'_i$.
Using~\eqref{Eq:bWa1}, we get
 \[
 1=b'_0\wedge b'_1\leq a'_0\vee a'_1
 =a^*_0\vee a^*_1\,,
 \]
a contradiction.
\end{proof}

\begin{lemma}\label{L:ConvRngLat}
Let~$C$ be a convex sublattice of a lattice~$L$ and let~$\gL$ be a chain.
Then the set~$M$, of all antitone $x\in L^{\gL}$ such that $x[\gL]\cap C\neq\es$, is a sublattice of~$L^{\gL}$.
\end{lemma}

\begin{proof}
Observe first that~$M$ is nonempty (for all constant maps with value in~$C$ belong to~$M$).

Now let $x,y\in M$.
We must prove that $x\vee y$ and~$x\wedge y$ both belong to~$M$.
The maps $x\vee y$ and~$x\wedge y$ are both antitone.
Let $\xi,\eta\in\gL$ and $u,v\in C$ such that $x_{\xi}=u$ and $y_{\eta}=v$.
Since~$\gL$ is a chain, we may assume that $\xi\leq\eta$.
Since $x_{\eta}\leq x_{\xi}=u$, we get $(x\vee y)_{\eta}=x_{\eta}\vee v\in[v,u\vee v]$, thus, as~$C$ is a convex sublattice of~$L$, $(x\vee y)_{\eta}\in C$.
Similarly, $(x\wedge y)_{\xi}\in C$.
\end{proof}

Our next lemma deals with solutions of equations in members of the variety~$\cN_5$\,.

For any nonnegative integer~$m$, we denote by~$\gO_m$ the set of all
tuples $z=(x_0,\dots,x_m,y_0,\dots,y_m,u,v)\in\sN_5^{2m+4}$ such that $x_0\leq\cdots\leq x_m$, $y_0\leq\cdots\leq y_m$, $u\wedge v\leq x_0\vee y_0$, and $(u\vee x_k\vee y_k)\wedge(v\vee x_k\vee y_k)\leq x_{k+1}\vee y_{k+1}$ whenever $0\leq k<m$.
In that context, we set $c(z)=u$, $d(z)=v$, $a_k(z)=x_k$\,, and $b_k(z)=y_k$\,, whenever $0\leq k\leq m$, and we denote by~$\seq{z}$ the sublattice of~$\sN_5$ generated by the entries of~$z$.
We also denote by~$0_z$ (resp., $1_z$) the least element (resp., the largest element) of~$\seq{z}$.
We set $L_m=\prod_{z\in\gO_m}\seq{z}$, and we denote by~$F_m$ the sublattice of~$L_m$ generated by $X_m=\set{a_0,\dots,a_m,b_0,\dots,b_m,c,d}$.
A standard argument of universal algebra shows that~$F_m$ is the lattice defined, within~$\cN_5$, by the generators~$a_0$, \dots, $a_m$, $b_0$, \dots, $b_m$, $c$, $d$ and the relations $a_0\leq\dots\leq a_m$, $b_0\leq\dots\leq b_m$, $c\wedge d\leq a_0\vee b_0$, and $(c\vee a_k\vee b_k)\wedge(d\vee a_k\vee b_k)\leq a_{k+1}\vee b_{k+1}$ whenever $0\leq k<m$.

We define~$\beta(t)$ as the least element $f\in F_m$ such that $t\leq f$, for any $t\in L_m$.
Observe that~$\beta$ is a surjective \jzh\ from~$L_m$ onto~$F_m$.
We also define~$\beta_0(t)$ as the meet, in~$L_m$, of all the elements of the generator set~$X_m$ above~$t$.
Hence, $\beta(t)\leq\beta_0(t)$.
Furthermore, if~$t$ is join-prime in~$L_m$, then $\beta(t)=\beta_0(t)$ (this is well known, see, for example, Freese, Je\v{z}ek, and Nation \cite[Theorem~2.4]{FJN}).

\begin{lemma}\label{L:NoMidFm}
For any nonnegative integer~$m$, there is no $f\in F_m$ satisfying the equation $(f\vee c)\wedge(f\vee d)=(f\wedge a_m)\vee(f\wedge b_m)$.
\end{lemma}

\begin{proof}
Let~$f$ satisfy the given equation.
Observe that $f=(f\vee c)\wedge(f\vee d)=(f\wedge a_m)\vee(f\wedge b_m)$.
For each $t\in\sN_5$ and each $I\subseteq[0,m]$, we define~$t\cdot I$ as the finite sequence $(t_0,\dots,t_m)$ defined by
 \[
 t_k=\begin{cases}
 0\,,&\text{if }k\notin I\,,\\
 t\,,&\text{if }k\in I\,,
 \end{cases}
 \qquad\text{for each }k\in[0,m]\,.
 \]
Furthermore, we denote by~$t\cdot z$ the element of~$L_m$ sending~$z$ to~$t$ and every $z'\neq z$ to~$0_{z'}$, whenever $z\in\gO_m$ and $t\in\seq{z}$.

\setcounter{claim}{0}

\begin{claim}\label{Cl:a0jjb0leqf}
$a_0\vee b_0\leq f$.
\end{claim}

\begin{cproof}
The $(2m+4)$-tuple $z=(p_1\cdot[0,m],p_3\cdot[0,m],p_2,p_2)$ belongs to~$\gO_m$, and $f(z)=f(z)\vee p_2=(f(z)\wedge p_1)\vee(f(z)\wedge p_3)$, thus $f(z)=1$, and thus, \emph{a fortiori}, $p_3\leq f(z)$.
The latter inequality means that $p_3\cdot z\leq f$, so, since $f\in F_m$, we get $\beta(p_3\cdot z)\leq f$.
Since~$p_3$ is join-prime in~$\seq{z}$, $p_3\cdot z$ is join-prime in~$L_m$, thus $\beta(p_3\cdot z)=\beta_0(p_3\cdot z)=b_0$.
Therefore, $b_0\leq f$.
Symmetrically, $a_0\leq f$.
\end{cproof}

\begin{claim}\label{Cl:akcleqf}
$(a_k\wedge c)\vee(a_k\wedge d)\vee(b_k\wedge c)\vee(b_k\wedge d)\leq f$,
whenever $0\leq k\leq m$.
\end{claim}

\begin{cproof}
We argue by induction on~$k$.
The case $k=0$ follows from Claim~\ref{Cl:a0jjb0leqf}.

Suppose the statement proved at~$k<m$.
Since $z=(p_3\cdot[k+1,m],p_1\cdot[k,m],p_2,p_3)$ belongs to~$\gO_m$ and $f(z)=(f(z)\vee p_2)\wedge(f(z)\vee p_3)=(f(z)\wedge p_1)\vee(f(z)\wedge p_3)$, we get $f(z)\in\set{0,1,p_3}$.
Moreover, by the induction hypothesis, $b_k\wedge c\leq f$, so we get $f(z)\geq p_1\wedge p_2=p_1$, so the only remaining possibility is $f(z)=1$, and so, \emph{a fortiori}, $p_3\leq f(z)$, that is, $p_3\cdot z\leq f$.
Hence, $a_{k+1}\wedge d=\beta_0(p_3\cdot z)=\beta(p_3\cdot z)\leq f$.
By symmetry, $a_{k+1}\wedge c$, $b_{k+1}\wedge c$, and $b_{k+1}\wedge d$ are also below~$f$.
\end{cproof}

Finally, the element $z=(0\cdot[0,m],p_1\cdot\set{m},p_2,p_3)$ belongs to~$\gO_m$, and
$f(z)=(f(z)\vee p_2)\wedge(f(z)\vee p_3)=f(z)\wedge p_1$, thus $f(z)=0$.
On the other hand, it follows from Claim~\ref{Cl:akcleqf} that $b_m\wedge c\leq f$, thus $p_1\leq f(z)$, a contradiction.
\end{proof}

\begin{theorem}\label{T:IdProjD}
Let~$\cV$ be a variety of lattices.
Then~$\DX$ is sharply transferable with respect to~$\cV$ if{f}~$\cV$ is contained in~$\cM_{\go}$\,.
\end{theorem}

\begin{proof}
We first prove that~$\DX$ is sharply transferable with respect to~$\cM_{\go}$\,.
By Theorem~\ref{T:Transf2Proj}, it suffices to prove that~$\DX$ is \idproj\ with respect to~$\cM_{\go}$\,.
Let $M\in\cM_{\go}$, let $\gf\colon\DX\to\nobreak\Id M$ be a lattice homomorphism, and let $f_0\colon\DX\to M$ be a choice function for~$\gf$.
We need to find a choice homomorphism $f\colon\DX\to M$ for~$\gf$ such that $f_0\leq f$.
We may assume that~$f_0$ is a \jh.
The ideals $\ba_i=\gf(\va_i)$ and $\bb_i=\gf(\vb_i)$, for $i\in\set{0,1}$, satisfy $\ba_0\vee\ba_1=\bb_0\cap\bb_1$.
Each element $b_i=f_0(\vb_i)$ belongs to~$\bb_i$\,, thus $b_0\wedge b_1$ belongs to $\bb_0\cap\bb_1=\ba_0\vee\ba_1$, and thus there are $a_i\geq f_0(\va_i)$ in~$\ba_i$\,, for $i\in\set{0,1}$, such that $b_0\wedge b_1\leq a_0\vee a_1$.
The element $b'_i=b_i\vee a_0\vee a_1$ belongs to~$\bb_i$\,, for each $i\in\set{0,1}$, thus
$b'_0\wedge b'_1$ belongs to $\bb_0\cap\bb_1=\ba_0\vee\ba_1$, and thus there are $a'_i\geq a_i$ in~$\ba_i$\,, for $i\in\set{0,1}$, such that $b'_0\wedge b'_1\leq a'_0\vee a'_1$.
Each element $a^*_i=a'_i\wedge b'_0\wedge b'_1$ belongs to~$\ba_i$\,, and it follows from Lemma~\ref{L:QuasiIdMgo} that $a^*_0\vee a^*_1=b'_0\wedge b'_1$.
The assignment $0_{\DX}\mapsto a^*_0\wedge a^*_1$, $\va_i\mapsto a^*_i$\,, $\vb_i\mapsto b'_i$\,, $1_{\DX}\mapsto b'_0\vee b'_1$ defines a choice homomorphism $f\colon\DX\to M$ for~$\gf$ with $f_0\leq f$.
Therefore, $\DX$ is sharply transferable with respect to~$\cM_{\go}$\,.

Conversely, let~$\cV$ be a lattice variety not contained in~$\cM_{\go}$\,.
Suppose that~$\DX$ is sharply transferable with respect to~$\cV$.

Suppose first that all lattices in~$\cV$ are modular.
It follows from J\'onsson~\cite{Jons68} (cf. Section~\ref{S:Basic}) that~$\sM_{3,3}$ belongs to~$\cV$.
Denote by~$\ZZ$ the chain of all integers.
By using the labeling of~$\sM_{3,3}$ introduced in Figure~\ref{Fig:MgoandM33}, together with Lemma~\ref{L:ConvRngLat}, we obtain that the sublattice~$M$ of~$\sM_{3,3}^{\ZZ}$, consisting of all antitone maps $x\colon\ZZ\to\sM_{3,3}$ such that $x[\ZZ]\cap\set{u,v}\neq\es$, belongs to~$\cV$.
For all $i\in\set{0,1}$ and all $n\in\ZZ$, we denote by~$a_{i,n}$ and~$b_{i,n}$ the elements of~$M$ defined, using the labeling of~$\sM_{3,3}$ represented in Figure~\ref{Fig:MgoandM33}, by
 \[
 a_{i,n}(k)=\begin{cases}
 u_i\,,&\text{if }k\leq n-1\,,\\
 u\,,&\text{if }k=n\,,\\
 0\,,&\text{if }k\geq n+1\,,
 \end{cases}\qquad
 b_{i,n}(k)=\begin{cases}
 1\,,&\text{if }k\leq n-2\,,\\
 v\,,&\text{if }k=n-1\,,\\
 v_i\,,&\text{if }k\geq n\,,
 \end{cases}
 \]
for every $k\in\ZZ$.
Observe that the sequences $\vecm{a_{i,n}}{n\in\ZZ}$ and $\vecm{b_{i,n}}{n\in\ZZ}$ are both ascending, for every $i\in\set{0,1}$.
Thus, the lower subsets~$\ba_i$ and~$\bb_i$\,, generated by the respective ranges of those sequences, are ideals of~$M$.
Furthermore, the verification of the following inequalities is straightforward:
 \begin{equation}\label{Eq:Ineqainbjn}
 b_{0,n}\wedge b_{1,n}\leq a_{0,n}\vee a_{1,n}
 \leq b_{0,n+1}\wedge b_{1,n+1}\,,\quad\text{whenever }n\in\ZZ\,. 
 \end{equation}
It follows that $\ba_0\vee\ba_1=\bb_0\cap\bb_1$.
Since~$\ba_0$ and~$\ba_1$ (resp., $\bb_0$ and~$\bb_1$) are incomparable with respect to set inclusion, there exists a unique lattice embedding $\gf\colon\DX\hookrightarrow\Id M$ such that $\gf(\va_i)=\ba_i$ and $\gf(\vb_i)=\bb_i$ for each $i\in\set{0,1}$.

For every $z\in M$, denote by~$z(\infty)$ the constant value of~$z(n)$, for large enough $n\in\ZZ$.

\begin{sclaim}
Let $i\in\set{0,1}$ and let $y\in\bb_i\setminus\bb_{1-i}$.
Then $y(\infty)=v_i$.
\end{sclaim}

\begin{scproof}
{}From $y\in\bb_i$ it follows that $y(\infty)\leq v_i$\,, so $y(\infty)\in\set{0,v_i}$.
Suppose that $y(\infty)=0$.
Since $y\in\bb_i$\,, there exists $m\in\ZZ$ such that $y\leq b_{i,m}$ and $y(k)=0$ whenever $k\geq m$.
It follows that $y\leq b_{1-i,m}$, so $y\in\bb_{1-i}$\,, a contradiction.
\end{scproof}

Since~$\DX$ is sharply transferable with respect to~$\cV$ and since $M\in\cV$, there are $x_i\in\ba_i\setminus\ba_{1-i}$ and $y_i\in\bb_i\setminus\bb_{1-i}$\,, for $i\in\set{0,1}$, such that $x_0\vee x_1=y_0\wedge y_1$.
{}From the Claim above it follows that $y_i(\infty)=v_i$.
Hence,
 \begin{equation}\label{Eq:yigeqvi}
 v_i\leq y_i(n)\quad\text{for every }n\in\ZZ\,.
 \end{equation}
On the other hand, from $x_i\in\ba_i$ it follows that
 \begin{equation}\label{Eq:xilequi}
 x_i(n)\leq u_i\quad\text{for every }n\in\ZZ\,.
 \end{equation}
Since $x_0(n)\vee x_1(n)=y_0(n)\wedge y_1(n)$ within~$\sM_{3,3}$\,, the only possibilities allowed by~\eqref{Eq:yigeqvi} and~\eqref{Eq:xilequi} above are that either $x_0(n)=x_1(n)=0$ and $y_i(n)=v_i$ for every $i\in\set{0,1}$, or $y_0(n)=y_1(n)=1$ and $x_i(n)=u_i$ for every $i\in\set{0,1}$ (inspect Figure~\ref{Fig:MgoandM33}).
In particular, the range of~$x_0$ is contained in $\set{0,u_0}$, in contradiction with $x_0\in M$.

Now suppose that~$\cV$ contains a nonmodular lattice.
It follows that $\cN_5\subseteq\cV$.
Denote by~$F$ the lattice defined, within~$\cN_5$, by the generators~$a_n$\,, $b_n$\,, for $n<\go$, and $c$, $d$ subjected to the relations $c\wedge d\leq a_0\vee b_0$, together with $a_n\leq a_{n+1}$\,, $b_n\leq b_{n+1}$\,, and $(c\vee a_n\vee b_n)\wedge(d\vee a_n\vee b_n)\leq a_{n+1}\vee b_{n+1}$ whenever $n<\go$.
Let~$\ba_0$, $\ba_1$, $\bb_0$, $\bb_1$ be the ideals of~$F$ generated by $\setm{a_n}{n<\go}$, $\setm{b_n}{n<\go}$, $\setm{c\vee a_n\vee b_n}{n<\go}$, $\setm{d\vee a_n\vee b_n}{n<\go}$, respectively.
The relations defining the~$a_n$\,, $b_n$\,, $c$, $d$ ensure that $\ba_0\vee\ba_1=\bb_0\cap\bb_1$.

Since, by our assumption, $\DX$ is sharply transferable with respect to~$\cN_5$, it is also \idproj\ with respect to~$\cN_5$ (cf. Theorem~\ref{T:Transf2Proj}), thus there are $x_i\in\ba_i$ and $y_i\in\bb_i$\,, for $i\in\set{0,1}$, such that $c\leq y_0$, $d\leq y_1$, and
 \begin{equation}\label{Eq:x0x1y0y1}
 x_0\vee x_1=y_0\wedge y_1
 \end{equation}
within~$F$.

Preservation of the canonical generators~$a_i$\,, $b_i$\,, $c$, $d$ defines lattice homomorphisms from~$F_k$ to~$F_l$ and from~$F_k$ to~$F$, whenever $k\leq l<\go$, and those homomorphisms form a direct system of lattices and lattice homomorphisms.
Since the direct limit (directed colimit in categorical language) $\varinjlim_{k<\go}F_k$ satisfies the universal property defining~$F$, it follows that $F=\varinjlim_{k<\go}F_k$.
Hence, there exists $m<\go$ such that~\eqref{Eq:x0x1y0y1} holds within~$F_m$ (we identify the~$x_i$ and~$y_i$ with lattice terms, with parameters from the~$a_j$\,, $b_j$\,, $c$, $d$, representing them).
Since each $x_i\in\ba_i$\,, we may, in addition, choose~$m$ in such a way that $x_0\leq a_m$ and $x_1\leq b_m$.
Set $f=x_0\vee x_1=y_0\wedge y_1$ (within~$F_m$).
{}From the inequalities $x_0\leq a_m\wedge f$, $x_1\leq b_m\wedge f$, $c\vee f\leq y_0$, and $d\vee f\leq y_1$, it follows that the equation
 \[
 (a_m\wedge f)\vee(b_m\wedge f)=(c\vee f)\wedge(d\vee f)
 \]
holds within~$F_m$, in contradiction with Lemma~\ref{L:NoMidFm}.
\end{proof}

A subalgebra~$A$ of a universal algebra~$B$ is \emph{pure} in~$B$, if whenever a finite equation system, with parameters from~$A$, has a solution in~$B$, it also has a solution in~$A$ (cf. Banaschewski and Nelson~\cite{BanNels72}).
An embedding $f\colon A\hookrightarrow B$ of universal algebras is pure, it~$f[A]$ is a pure subalgebra of~$B$.
Nelson observes in~\cite{Nels74} that as a consequence of Theorem~\ref{T:IdProjDLat}, the canonical embedding, of any distributive lattice into its ideal lattice, is pure.
She also finds an example (constructed as the dual lattice of an example from Wille~\cite{Wille74}) of a lattice of which the canonical embedding into its ideal lattice is not pure, and she asks whether this can be done for modular lattices.
The following result answers that question in the negative.

\begin{theorem}\label{T:Notpure}
There exists a lattice $M\in\cM_{3,3}$ such that the canonical embedding from~$M$ into~$\Id M$ is not pure.
\end{theorem}

Of course, since~$\sM_{3,3}$ is modular, so is~$M$.

\begin{proof}
As in the second part of the proof of Theorem~\ref{T:IdProjD}, $M$ is the sublattice of~$\sM_{3,3}^{\ZZ}$ consisting of all antitone maps $x\colon\ZZ\to\sM_{3,3}$ such that $x[\ZZ]\cap\set{u,v}\neq\es$.
The elements~$a_{i,n}$ and~$b_{i,n}$ of~$M$, and the ideals~$\ba_i$ and~$\bb_i$\,, for $i\in\set{0,1}$ and $n\in\ZZ$, are defined as in the proof of Theorem~\ref{T:IdProjD}.
The constant map $\ol{x}\colon\ZZ\to\set{x}$ belongs to~$M$ if{f} $x\in\set{u,v}$, for every $x\in\sM_{3,3}$\,.
Let~$\dnw y$ be shorthand for $M\dnw y$, whenever $y\in M$.

\setcounter{claim}{0}

The proofs of our next two claims are straightforward computations.

\begin{claim}\label{Cl:IdealSys}
The following inequalities hold in $\Id M$:
 \begin{gather*}
 \dnw a_{i,0}\leq\ba_i\text{ and }\dnw b_{i,1}\leq\bb_i\,,
 \quad\text{for all }i\in\set{0,1}\,;\\
 \ba_0\vee\ba_1=\bb_0\wedge\bb_1\,;\\
 \ba_0\vee\dnw\ol{v}=\ba_1\vee\dnw\ol{v}=\bb_0\vee\bb_1\,;\\
 \bb_0\wedge\dnw\ol{u}=\bb_1\wedge\dnw\ol{u}=\ba_0\wedge\ba_1\,.
 \end{gather*}
\end{claim}

\begin{claim}\label{Cl:ResSys}
The only quadruples $(x_0,x_1,y_0,y_1)\in\sM_{3,3}^4$ such that
 \begin{gather*}
 x_0\vee x_1=y_0\wedge y_1\,;\\
 x_0\vee v=x_1\vee v=y_0\vee y_1\,;\\
 y_0\wedge u=y_1\wedge u=x_0\wedge x_1 
 \end{gather*}
are $(u_0,u_1,1,1)$, $(u_1,u_0,1,1)$, $(0,0,v_0,v_1)$, $(0,0,v_1,v_0)$, $(u,u,u,v)$, $(u,u,v,u)$,\linebreak $(u,v,v,v)$, $(v,u,v,v)$.
\end{claim}

Now suppose that there is a quadruple $z=(x_0,x_1,y_0,y_1)\in M^4$ that satisfies the system of inequalities given in Claim~\ref{Cl:IdealSys}; that is,
 \begin{gather}
 a_{i,0}\leq x_i\text{ and }b_{i,1}\leq y_i\,,\quad\text{for all }i\in\set{0,1}\,;
 \label{Eq:ai0bi1xiyi}\\
 x_0\vee x_1=y_0\wedge y_1\,;\label{Eq:x01y01}\\
 x_0\vee\ol{v}=x_1\vee\ol{v}=y_0\vee y_1\,;\label{Eq:x01vy01}\\
 y_0\wedge\ol{u}=y_1\wedge\ol{u}=x_0\wedge x_1\,.\label{Eq:y01ux01}
 \end{gather}

\begin{claim}\label{Cl:z(<0)}
$z(n)=(u_0,u_1,1,1)$, for every $n<0$.
\end{claim}

\begin{cproof}
For each $i\in\set{0,1}$, it follows from~\eqref{Eq:ai0bi1xiyi} that $u_i=a_{i,0}(n)\leq x_i(n)$.
Since the quadruple $z(n)=(x_0(n),x_1(n),y_0(n),y_1(n))$ satisfies the equation system given in Claim~\ref{Cl:ResSys}, this leaves the only possibility $z(n)=(u_0,u_1,1,1)$.
\end{cproof}

\begin{claim}\label{Cl:z(>0)}
$z(n)=(0,0,v_0,v_1)$, for every $n>0$.
\end{claim}

\begin{cproof}
For each $i\in\set{0,1}$, it follows from~\eqref{Eq:ai0bi1xiyi} that $v_i=b_{i,1}(n)\leq y_i(n)$.
Since the quadruple $z(n)=(x_0(n),x_1(n),y_0(n),y_1(n))$ satisfies the equation system given in Claim~\ref{Cl:ResSys}, it can only take one of the values
 \[
 (u_0,u_1,1,1)\,,\ (u_1,u_0,1,1)\,,\ (0,0,v_0,v_1)\,,\ (u,v,v,v)\,,\, (v,u,v,v)\,.
 \]
Moreover, since each~$x_i$ and~$y_i$ is antitone and by Claim~\ref{Cl:z(<0)}, we get $z(n)\leq(u_0,u_1,1,1)$.
Hence either $z(n)=(u_0,u_1,1,1)$ or $z(n)=(0,0,v_0,v_1)$.
In the first case, it follows from the isotonicity of~$z$ that $z(-1)\leq z(0)\leq z(1)$, thus $z(0)=(u_0,u_1,1,1)$, and thus $z(k)=(u_0,u_1,1,1)$ for all~$k$.
In particular, $x_0=\ol{u_0}$ does not belong to~$M$, a contradiction.
\end{cproof}

Since $z(-1)\leq z(0)\leq z(1)$ and by Claims~\ref{Cl:z(<0)} and~\ref{Cl:z(>0)}, we obtain the inequalities
 \[
 (0,0,v_0,v_1)\leq z(0)\leq(u_0,u_1,1,1)\,.
 \]
By Claim~\ref{Cl:ResSys}, it follows that~ $z(0)$ takes one of the values
 \[
 (u_0,u_1,1,1)\,,\ (0,0,v_0,v_1)\,.
 \]
By Claims~\ref{Cl:z(<0)} and~\ref{Cl:z(>0)}, it follows that $x_0(n)\in\set{0,u_0}$ for every $n\in\ZZ$, in contradiction with $x_0\in M$.

We have thus proved that the system of inequalities \eqref{Eq:ai0bi1xiyi}--\eqref{Eq:y01ux01}, with unknowns~$x_0$, $x_1$, $y_0$, $y_1$ and parameters~$a_{0,0}$\,, $a_{1,0}$\,, $b_{0,1}$\,, $b_{1,1}$\,, $\ol{u}$, $\ol{v}$, has no solution in~$M^4$.
\end{proof}

\begin{remark}\label{Rk:Notpure}
Although the lattice~$M$ of the proof of Theorem~\ref{T:Notpure} is not bounded, it embeds as a convex sublattice into the bounded lattice $M_{01}=M\cup\set{0,1}$, for a new bottom element~$0$ and a new top element~$1$.
The lattice~$M_{01}$ belongs to~$\cM_{3,3}$\,, and it satisfies the same negative property as~$M$:
the system of inequalities \eqref{Eq:ai0bi1xiyi}--\eqref{Eq:y01ux01} has no solution in~$M^4$, thus also no solution in $(M_{01})^4$.
\end{remark}

It follows from Theorem~\ref{T:IdProjD} that the lattice~$\DX$ is not sharply transferable with respect to the variety~$\cM$ of all modular lattices.
This can be put in contrast with the following result, which implies that transferability and sharp transferability, with respect to~$\cM$, are distinct concepts, even for finite distributive lattices.

\begin{proposition}\label{P:D4transf}
The lattice~$\DX$ is transferable with respect to~$\cM$.
\end{proposition}

\begin{proof}
An element~$c$ in a lattice~$M$ is \emph{doubly reducible} if there are $a_0,a_1<c$ and $b_0,b_1>c$ such that $c=a_0\vee a_1=b_0\wedge b_1$.
Obviously, $M$ contains~$\DX$ as a sublattice if{f}~$M$ has a doubly reducible element.
We need to prove that if~$M$ is modular and~$\Id M$ has a doubly reducible element, then~$M$ has a doubly reducible element.

Suppose, to the contrary, that~$M$ has no doubly reducible element.
By assumption, there are incomparable pairs $(\ba_0,\ba_1)$ and $(\bb_0,\bb_1)$ of ideals of~$M$ such that $\ba_0\vee\ba_1=\bb_0\cap\bb_1$.
Pick $b_i\in\bb_i\setminus\bb_{1-i}$\,, for $i\in\set{0,1}$.
Since $b_0\wedge b_1\in\bb_0\cap\bb_1=\ba_0\vee\ba_1$, there is $(a_0,a_1)\in\ba_0\times\ba_1$ such that $b_0\wedge b_1\leq a_0\vee a_1$.
By further enlarging the~$a_i$\,, we may assume that $a_i\notin\ba_{1-i}$ whenever $i\in\set{0,1}$.
Set $a=a_0\vee a_1$.
The element $b'_i=b_i\vee a$ belongs to~$\bb_i\setminus\bb_{1-i}$\,, for each $i\in\set{0,1}$;
in particular, $b'_0$ and~$b'_1$ are incomparable.
If $b'_0\wedge b'_1=a$ ($=a_0\vee a_1$) then~$a$ is doubly reducible, a contradiction.
It follows that $a<b'_0\vee b'_1$.
Since $b_0\wedge b_1\leq a$, it follows that $b_i<b'_i$ for some $i\in\set{0,1}$, that is, $a\nleq b_i$; say, $a\nleq b_1$.
Further, since $a\leq b'_0$, it follows from the modularity of~$M$ that
 \begin{equation}\label{Eq:b'0mmb'1}
 b'_0\wedge b'_1=b'_0\wedge(b_1\vee a)=(b'_0\wedge b_1)\vee a\,.
 \end{equation}
Since~$b'_0$ and~$b'_1$ are incomparable and since~$M$ has no doubly reducible element, it follows from~\eqref{Eq:b'0mmb'1} that either $b'_0\wedge b_1\leq a$ or $a\leq b'_0\wedge b_1$.
In the first case, it follows from~\eqref{Eq:b'0mmb'1} that $b'_0\wedge b'_1=a$, a contradiction.
In the second case, we get $a\leq b_1$, a contradiction.
\end{proof}

\section{Partial lattices satisfying Whitman's Condition}\label{S:DavSands}

Whitman's Condition~(W) can be defined for partial lattices the same way it is defined for lattices: namely, a partial lattice~$P$ satisfies~(W) if for all nonempty finite subsets~$U$ and~$V$ of~$P$, if~$\bigwedge U$ and~$\bigvee V$ are both defined and $\bigwedge U\leq\bigvee V$, then either there is $u\in U$ such that $u\leq\bigvee V$ or there is $v\in V$ such that $\bigwedge U\leq v$.

\begin{example}\label{Ex:Diamond}
Whenever~$n$ is a positive integer, the \emph{$n$-cube} is the  powerset lattice~$\sB_n$ of an $n$-element set, and the \emph{$n$-diamond}%
\footnote{Although some references call~$\sP_n$ the \emph{$(n-1)$-diamond}, it seems that the current usage shows a slight preference towards the term ``$n$-diamond'', see, for example, Jipsen and Rose~\cite[Section~3.3]{JiRo92}.}
is defined as the partial lattice $\sP_n=\sB_n\cup\set{e}$, for an element $e\notin\sB_n$ subjected to the relations $a\wedge e=0$ and $a\vee e=1$ whenever~$a$ is an atom of~$\sB_n$.

It is easy to see that~$\sP_n$ satisfies~(W) if{f}~$\sB_n$ satisfies~(W) if{f} $n\leq3$.
\end{example}

The following result, implicit in Huhn~\cite{Huhn72}, is explicitly stated in Freese~\cite{Freese76}.

\begin{theorem}[Huhn 1972; Freese 1976]\label{T:DiamProj}
The $n$-diamond~$\sP_n$ is projective with respect to the variety~$\cM$ of all modular lattices, for every positive integer~$n$.
\end{theorem}

A direct application of Theorem~\ref{T:Transf2Proj} thus yields the following.

\begin{corollary}\label{C:DiamProj}
The $n$-diamond~$\sP_n$ is sharply transferable with respect to~$\cM$, for every positive integer~$n$.
\end{corollary}

Whitman's Condition plays a crucial role in the following result, established for lattices in Davey and Sands \cite[Theorem~1]{DavSan}.
We include a proof for convenience.

\begin{proposition}[Davey and Sands, 1977]\label{P:DavSands}
Every partial lattice with~\textup{(W)} is projective with respect to the class of all lattices without infinite chains.
\end{proposition}

\begin{proof}
Let~$P$ be a partial lattice with~(W), let~$K$ and~$L$ be lattices without infinite chains, let $h\colon L\twoheadrightarrow K$ be a surjective lattice homomorphism, and let $f\colon P\to K$ be a homomorphism of partial lattices.
Since~$L$ has no infinite descending sequence, the lower adjoint ($\beta\colon K\hookrightarrow L$, $x\mapsto\min h^{-1}\set{x}$) of~$h$ is defined.
The set~$\cC$, of all partial \jh{s} $\xi\colon P\to L$ such that $f=h\circ\xi$, contains~$\beta\circ f$ as an element.
Since~$L$ has no infinite ascending sequence and by Zorn's Lemma, $\cC$ has a maximal element~$g$ such that $\beta\circ f\leq g$.
We claim that~$g$ is a homomorphism of partial lattices.
Suppose otherwise.
Then~$g$ is not a partial \mh, that is, there are $a\in P$ and $U\in\fine{P}$ such that $a=\bigwedge U$ and $g(a)<\bigwedge g[U]$.
Set $e=\bigwedge g[U]$ and define $g'\colon P\to L$ by setting
 \[
 g'(x)=\begin{cases}
 g(x)&(\text{if }a\nleq x)\\
 g(x)\vee e&(\text{if }a\leq x)
 \end{cases}\,,\qquad\text{for any }x\in P\,.
 \]
Observe that $h\circ g'=g$.
Since $g\leq g'$ and $g'(a)=e>g(a)$, it follows from the maximality assumption of~$g$ that~$g'$ is not a \jh, that is, there are $b\in P$ and $V\in\fine{P}$ such that $b=\bigvee V$ and $\bigvee g'[V]<g'(b)$.
Necessarily, $a\leq b$, $a\nleq v$ for any $v\in V$, and $u\nleq b$ for any~$u\in U$.
Since~$P$ satisfies~(W), this is a contradiction.
\end{proof}

Recall that a variety~$\cV$ of algebras is \emph{locally finite} if every finitely generated algebra in~$\cV$ is finite.

\begin{corollary}\label{C:DavSands1}
Let~$P$ be a finite partial lattice with~\textup{(W)} and let~$\cV$ be a locally finite lattice variety.
Then~$P$ is both projective and \idproj\ with respect to~$\cV$.
\end{corollary}

\begin{proof}
By Theorem~\ref{T:Transf2Proj}, it suffices to prove that~$P$ is projective with respect to~$\cV$.
Let $K,L\in\cV$, let $h\colon L\twoheadrightarrow K$ be a surjective lattice homomorphism, and let $f\colon P\to\nobreak K$ be a homomorphism of partial lattices.
Since~$P$ is finite and~$K$ is locally finite, the sublattice~$K'$ of~$K$ generated by~$f[P]$ is finite.
Since~$L$ is locally finite, there is a finite sublattice~$L'$ of~$L$ such that $h[L']=K'$.
By applying Proposition~\ref{P:DavSands} to the domain-range restriction of~$h$ from~$L'$ onto~$K'$, we obtain a homomorphism $g\colon P\to L'$ of partial lattices such that $f=h\circ g$.
\end{proof}

For example, the $3$-diamond $\sP_3$ is both projective and sharply transferable with respect to any locally finite lattice variety.

\begin{corollary}\label{C:DavSands3}
A finite lattice~$P$, satisfying~\textup{(W)}, is semidistributive if{f} there is a largest lattice variety~$\cV$ such that~
$P$ is projective \pup{resp., \idproj} with respect to~$\cV$.
In that case, $\cV=\cL$, the variety of all lattices.
\end{corollary}

\begin{proof}
If~$P$ is semidistributive, then, since it is finite and satisfies~(W), it is both projective and sharply transferable (this is mostly due to Nation~\cite{Nation82}; see Theorem~\ref{T:CharTransf}).
By Theorem~\ref{T:Transf2Proj}, it follows that~$P$ is both projective and \idproj\ with respect to~$\cL$.

Suppose, conversely, that there is a largest variety~$\cV$ such that~$P$ is projective (resp. \idproj) with respect to~$\cV$.
It follows from Corollary~\ref{C:DavSands1} that~$P$ is both projective and \idproj\ with respect to the variety generated by any finite lattice.
Since finite lattices generate the variety of all lattices (cf. Dean~\cite{Dean56b}), we get $\cV=\cL$.
Using Theorem~\ref{T:Transf2Proj}, it follows that~$P$ is projective (resp., sharply transferable), thus semidistributive (cf. Theorem~\ref{T:CharTransf}).
\end{proof}

For example, the lattice~$\sM_3$\,, of length~$2$ and with three atoms, satisfies~(W), but it is not semidistributive.
Hence, there is no largest variety~$\cV$ such that~$\sM_3$ is projective (resp., \idproj) with respect to~$\cV$.

We conclude the paper with the following problems.

\begin{problem}\label{Pb:ProjVar}
What are the finite lattices~$P$ for which there is a largest lattice variety~$\cV$ such that~$P$ is \idproj\ (resp., projective) with respect to~$\cV$?
Does that class include all finite distributive lattices?
What are the possible values of~$\cV$?
\end{problem}

For example, it follows from Theorem~\ref{T:IdProjD} that the specialization of Problem~\ref{Pb:ProjVar} to $P=\DX$ has the solution $\cV=\cM_{\go}$\,.
On the other hand, if~$P$ is projective with respect to~$\cL$ (e.g., $P$ is Boolean with three atoms), then the solution to our problem is $\cV=\cL$, the variety of all lattices.
On the other hand, by the above, if $P=\sM_3$\,, then there is no largest variety~$\cV$ such that~$P$ is \idproj\ (resp., projective) with respect to~$\cV$.

By Theorem~\ref{T:IdProjDLat}, every finite distributive lattice is sharply transferable with respect to~$\cD$.
Moreover, by Theorem~\ref{T:IdProjD}, the distributive lattice~$\DX$ is sharply transferable with respect to the even larger variety~$\cM_{\go}$ generated by all lattices of length two.
This suggests the following problem.

\begin{problem}\label{Pb:DMgoidproj}
Is every finite distributive lattice sharply transferable with respect to~$\cM_{\go}$?
\end{problem}


\providecommand{\noopsort}[1]{}\def\cprime{$'$}
  \def\polhk#1{\setbox0=\hbox{#1}{\ooalign{\hidewidth
  \lower1.5ex\hbox{`}\hidewidth\crcr\unhbox0}}}
\providecommand{\bysame}{\leavevmode\hbox to3em{\hrulefill}\thinspace}
\providecommand{\MR}{\relax\ifhmode\unskip\space\fi MR }
\providecommand{\MRhref}[2]{%
  \href{http://www.ams.org/mathscinet-getitem?mr=#1}{#2}
}
\providecommand{\href}[2]{#2}

\end{document}